\documentclass[12pt, reqno]{amsart}

\usepackage{graphicx}
\usepackage{amssymb}
\usepackage{amsthm}
\usepackage{listings}
\usepackage{lineno}
\usepackage[margin=3cm]{geometry}
\usepackage[all,cmtip, color,matrix,arrow]{xy}
\usepackage{subfig}

\usepackage{amsmath}%To use \text 
\usepackage[utf8]{inputenc}
\usepackage{hyperref}
\usepackage[capitalize]{cleveref}  
\crefname{thm}{Theorem}{Theorems}
\usepackage{bbold}
\usepackage[export]{adjustbox}
\usepackage{todonotes}
\usepackage{bm}
\usepackage{wrapfig}
\usepackage{bbold}
\usepackage{float}
\usepackage{mathtools}
\usepackage{aliascnt}
\newaliascnt{eqfloat}{equation}
\newfloat{eqfloat}{h}{eqflts}
\floatname{eqfloat}{Equation}
\usepackage{dirtytalk}

\newcommand*{\ORGeqfloat}{}
\let\ORGeqfloat\eqfloat
\def\eqfloat{%
  \let\ORIGINALcaption\caption
  \def\caption{%
    \addtocounter{equation}{-1}%
    \ORIGINALcaption
  }%
  \ORGeqfloat
}

\theoremstyle{definition}
\newtheorem{thm}{Theorem}[section]
\newtheorem{prop}[thm]{Proposition}
\newtheorem{lm}[thm]{Lemma}
\newtheorem{cor}[thm]{Corollary}
\newtheorem{obs}[thm]{Observation}
\newtheorem{defin}[thm]{Definition}
\newtheorem{smpl}[thm]{Example}

\crefname{lm}{Lemma}{Lemmas}
\crefname{thm}{Theorem}{Theorems}
\crefname{prop}{Proposition}{Propositions}
\crefname{defin}{Definition}{Definitions}
\crefname{rem}{Remark}{Remarks}

\newcommand{\oPi}{\mathbf{C}}
\newcommand{\opi}{\vec{\boldsymbol{\pi}}}

\newcommand{\III}{\vec{\mathbf{I}}}
\newcommand{\JJJ}{\vec{\mathbf{J}}}

\DeclareMathOperator{\id}{id}

\DeclareMathOperator{\pat}{\mathbf{pat}}

\DeclareMathOperator{\inc}{\mathrm{inc}}

\DeclareMathOperator{\sFunc}{\mathrm{SurFunc}}
\DeclareMathOperator{\End}{\mathrm{End}}
\newcommand{\Fset}{\mathsf{Set^{\times}}}
 
\newcommand{\Vect}{\mathsf{Vect}}
\newcommand{\gVec}{\mathsf{gVec}}
\newcommand{\Set}{\mathsf{Set}}
\newcommand{\Ss}{\mathsf{Sp}} % set species
\newcommand{\Ssk}{\mathsf{Sp}_\Kb} %vector species
\newcommand{\Spr}{\mathsf{Spr}} % set species with restrictions

\newcommand{\Kb}{\mathbb{K}}

\newcommand{\rP}{\mathrm{P}}
\newcommand{\rQ}{\mathrm{Q}}
\newcommand{\rH}{\mathrm{H}}

\newcommand{\prP}{\mathtt{P}}
\newcommand{\prQ}{\mathtt{Q}}

\newcommand{\prR}{\mathtt{R}}

\newcommand{\thh}{\mathbf{h}}

\newcommand{\tp}{\mathbf{p}} 
\newcommand{\tq}{\mathbf{q}}

\newcommand{\trr}{\mathbf{r}} 
\newcommand{\Kc}{\mathcal{K}}
\newcommand{\Kcb}{\overline{\Kc}}

\newcommand{\sq}{{{\scriptstyle{\square}}}}

\usepackage{amsaddr}

\begin{document}

\title{Antipode formulas for pattern Hopf algebras} % Subtitle

%\author{Raul Penaguiao\footnote{\href{mailto:raulpenaguiao@sfsu.edu}{raulpenaguiao@sfsu.edu}}\footnote{Institute of Mathematics, University of Zurich, Winterthurerstrasse 190, Zurich, CH - 8057.}\footnote{{\bf Keywords:} marked permutations, presheaves, species, Hopf algebras, free algebras}\footnote{2010 AMS Mathematics Subject Classification 2010: 05E05, 16T05, 18D10}}

\author{Raul Penaguiao, Yannic Vargas}
\email{raulpenaguiao@sfsu.edu}
\email{yvargaslozada@tugraz.at}
\address{San Francisco State University}
\address{Technische Universit\"at Graz}
\keywords{permutations, presheaves, species with restrictions, species, Hopf algebras, free algebras, antipode, cancellation-free, chromatic, reciprocity}
\subjclass[2010]{05E05, 16T05, 18D10}
\date{\today} % Date

\begin{abstract}
The permutation pattern Hopf algebra is a commutative filtered and connected Hopf algebra.
Its product structure stems from counting patterns of a permutation, interpreting the coefficients as permutation quasi-shuffles.
The Hopf algebra was shown to be a free commutative algebra and to fit into a general framework of pattern Hopf algebras, via species with restrictions.

In this paper we introduce the cancellation-free and grouping-free formula for the antipode of the permutation pattern Hopf algebra.
To obtain this formula, we use the popular sign-reversing involution method, by Benedetti and Sagan.
This formula has applications on polynomial invariants on permutations, in particular for obtaining reciprocity theorems.
On our way, we also introduce the packed word patterns Hopf algebra and present a formula for its antipode.

Other pattern algebras are discussed here, notably on parking functions, which recovers notions recently studied by Adeniran and Pudwell, and by Qiu and Remmel.
\end{abstract}

\maketitle

\tableofcontents

\section{Introduction}

In his now celebrated  ``Lemma 14'', Takeuchi (see \cite{Takeuchi1971}, Lemma 14) obtained a quite general formula for the antipode of a Hopf algebra. 
This is an antipode formula for any filtered Hopf algebra, which can be applied in much generality and fits the framework of \textbf{pattern Hopf algebras}, which we introduce below.
However, it has been observed that it is not the most economical formula, as it leaves some cancellations to be made.

\

Economical formulas for antipodes in Hopf algebras in combinatorics have played an important role in extracting old and new combinatorial equations, see \cite{Schmitt1993, humpert2012incidence, BS2017, aguiar2017hopf, xu2022cancellation}.
In particular, Humpert and Martin were able to explain, in \cite{humpert2012incidence}, an elusive \textit{reciprocity relation} on graphs first presented by Stanley in \cite{stanley1975combinatorial}.
We will introduce the reciprocity relation now.

\

On a graph $G$, we define the chromatic function by counting, for each $n\geq 0$, the number of \textit{stable colourings} of its vertices, that is, colourings such that each edge is monochromatic.
This function, $\chi_G$, turns out to be a polynomial, and its degree is the number of vertices of $G$.
It makes sense to explore the evaluation $\chi_G(-1)$, and Stanley proved that this counts the \textit{acyclic orientations} of $G$.
The observation of Humpert and Martin is that $\chi$ is in fact a Hopf algebra morphism from the \textbf{incidence Hopf algebra on graphs} to the polynomial Hopf algebra, so it commutes with the antipodes of each Hopf algebra.
The antipode of the polynomial Hopf algebra is $S(x) = -x$, so the chromatic polynomial on negative numbers becomes easier to compute, via
$$\chi_G(-x) = \chi_{S(G)}(x)\, .$$

A simple formula for the antipode $S(G)$ on the incidence Hopf algebra on graphs was also presented by Humpert and Martin, where the number of acyclic orientations plays a role and explains the previous result from Stanley.
This antipode formula was also used to obtain new relations involving the Tutte polynomial.

\

A method for obtaining a cancellation-free formula that seem to work with a large family of Hopf algebras was brought forth by Sagan and Benedetti in \cite{BS2017}, called \textit{sign-reversing involution method}.
This is a classical method in enumerative combinatorics, that has found applications in fields as far as number theory e.g. in \cite{zagier2009one}.
There, it was shown that $x^2+y^2 = p$ has integer solutions for $p$ prime whenever $p\equiv_4 1$.
This is a classical fact, but this new proof uses an involution method.

\

In Hopf algebras, this method requires careful treatment of the antipode formula of Takeuchi.
However, it varies widely depending on the combinatorial Hopf algebra at hand.
%Therefore, for each Hopf algebra, a new and original application of the method is needed.
Notwithstanding, this has been shown to work in the shuffle Hopf algebra, the incidence Hopf algebra on graphs and the Hopf algebras of quasisymmetric functions and multi-quasisymmetricfunctions (see \cite{BS2017}). It is still a challenging problem to find sign-reversing involutions to find cancellative-free formulas for the antipode of others classical Hopf algebras.  \footnote{In \cite{MalvenutoReutenauer}, the antipode of the Malvenuto–Reutenauer Hopf algebra was partially computed}

\

In \cite{aguiar2017hopf}, a cancellation-free antipode formula for a Hopf structure on generalized permutahedra was found. 
This has striking consequences, because several interesting combinatorial structures can be embedded in the Hopf algebra of the generalized permutahedra.
Specific examples are graphs, matroids, posets and set partitions.
Thus, this allows us to readily determine a cancellation-free formula for all these Hopf algebras.

\

Parallel to this development is the study of permutation patterns.
This is a study with roots in computer science, pioneered by Knuth in \cite{Knuth}, where a description for the \textit{stack sortable} permutations was presented via permutation patterns.
In the meantime, permutation patterns has become a well established area of expertise in combinatorics, see \cite{linton2010permutation}.

\

The second author, in \cite{Vargas}, introduced yet another tool to study permutation patterns, by building the permutation pattern Hopf algebra.
Specifically, if we consider finite sums of functions of the form 
$$ \binom{\sigma}{\pi} \coloneqq \pat_{\pi}(\tau)\coloneqq  \#\{\text{ ways to fit $\pi$ in $\tau$ }\}\, ,$$
the central functions in the study of permutation patterns, we span a vector space that is closed for pointwise product, see \eqref{eq:prodperm}.
The algebra corresponding to this pointwise product is the \textbf{permutation pattern Hopf algebra} $\mathcal{A}(\mathtt{Per})$, and is shown to be free in \cite{Vargas}.

\

The construction was generalized to other combinatorial objects, in \cite{Penaguiao2020}, as long as there is a notion of restrictions, for instance graphs or marked permutations.
The resulting structures are called \textbf{pattern Hopf algebras}.
It is conjectured that all pattern Hopf algebras are free.

\

In this paper we provide a cancellation-free and grouping-free antipode formula for the pattern Hopf algebras on permutations and on packed words.
This is an application of the sign-reversing involution method.

\

We also present an original species with restrictions structure on parking functions.
This is part of a project to interpret common combinatorial objects as species with restrictions.
Crucially, objects that do not have an inherent labelling are not amenable to an interpretation as species, so this codifies an important step in this project.
The extra step here is done with a help of a bijection between labelled Dyck paths and parking functions.

\

The notion of patterns in parking functions here presented recovers the one presented recently in \cite{adeniran2022pattern}.
There, the authors study the number of parking functions that avoid a set of five parking functions of size three, recovering sequences like the Catalan numbers.
The original introduction of patterns in parking functions hearkens back to \cite{qiu2018patterns}.
Parking functions themselves are a recent endeavour in combinatorics, being introduced for the first time in \cite{konheim1966occupancy}, where it was shown that there are $(n+1)^{n-1}$ parking functions of length $n$.

\

In \cref{sec:antipode_computing}, we present an example of an application of the sign-reversing involution formula.
We also present some examples of some antipode computations via \textit{ad hoc} methods. An introduction to the basic notions of Joyal species and algebraic structures related to these are presented in \cref{sec:species}. In \cref{sec:pattern_algebra_contruction}, we present the algebra and category theory background to pattern Hopf algebras. We recall the pattern Hopf algebra construction from \cite{Penaguiao2020}, along with its product and coproduct structure, from a species with restrictions, so that this article is self contained.
In \cref{sec:species_restrictions}, we present several examples of species with restrictions, centrally the one on permutations, but also an original one in parking functions.
%We also present here species with restrictions on Dyck paths and on parking functions, and present simple examples of patterns in parking functions.
In \cref{sec:formula_general,sec:formula_pp}, we present the main result, the cancellation-free and grouping-free formula for the antipode in packed words and in $\mathcal{A}(\mathtt{Per})$.
%In this section we start by describing the sign-reversing involution method to obtain a cancellation-free formula for the antipode of a Hopf algebra.
%Then we present the main result of this section in \cref{thm:antipode_perms_intro}, a cancellation-free formula for the antipode of $\mathcal{A}(\mathtt{Per})$.
\subsection{The permutation pattern Hopf algebra and the main result}

In this section we introduce the Hopf algebra structure on permutation patterns, and present the main result on this paper.

\

Let $\pat_{\pi}$ be a function on permutations, so that $\pat_{\pi}(\sigma)$ counts the number of restrictions of the permutation $\sigma$ that fit the pattern $\pi$.
In this way, the collection of permutation pattern functions $\{\pat_{\pi}\}$ is linearly independent, so it is a basis of a vector space $\mathcal A (\mathtt{Per})$.
However, in \cite{Vargas}, it was shown that the pointwise product of two such functions can be expressed as a sum of other permutation pattern functions:
\begin{equation}\label{eq:prodperm}
\pat_{\pi_1} \pat_{\pi_2} = \sum_{\sigma} \binom{\sigma}{\pi_1, \pi_2} \pat_{\sigma} \, .
\end{equation}

The coefficients $\binom{\sigma}{\pi_1, \pi_2}$ that arise in this product formula count the so called \textbf{quasi-shuffle signatures}, or \textbf{QSS}, of $\sigma$ from $\pi_1, \pi_2$.
Let us precise this definition:

\begin{defin}[QSS on permutations]
A QSS of $\sigma$ from $\pi_1, \dots, \pi_n$ is a tuple $\III = (I_1, \dots, I_n)$ of sets on the ground set of the permutation $\sigma$, that cover this ground set and that the restricted permutation $\sigma|_{I_i}$ is exactly a pattern of $\pi_i$.

\

Note how any reordering of a QSS is also a QSS.
It was shown in \cite{Penaguiao2020} that $\binom{\sigma}{\pi_1, \dots, \pi_n}$, the coefficient that arises in the iterated product of $n$ elements, counts the number of ways of covering $\sigma$ with $n$ permutations, each fitting the patterns $\pi_1, \dots, \pi_n$.
\end{defin}

\

This algebra $\mathcal A(\mathtt{Per})$ can be endowed with a Hopf algebra structure with the help of the diagonal sum of permutations, $\oplus$, also called shifted concatenation of permutations.
If one lets $\pi = \pi_1 \oplus \dots \oplus \pi_n$ be the decomposition of $\pi$ into $\oplus$-indecomposable permutations under the $\oplus$ product, we define the shifted deconcatenation coproduct
$$\Delta \pat_{\pi} = \sum_{k=0}^n \pat_{\pi_1\oplus \dots \oplus \pi_k} \otimes \pat_{\pi_{k+1}\oplus \dots \oplus \pi_n} \in \mathcal A (\mathtt{Per}) \otimes \mathcal A (\mathtt{Per})\, .$$

To present the antipode formula we need to refine the notion of QSS.
\begin{defin}[Interlacing QSS on permutations]
A QSS $\III = (I_1, \dots, I_n)$ is said to be \textbf{non-interlacing} if there is some $i=1, \dots, n-1$ such that $I_i < I_{i+1}$ and $\sigma(I_i) < \sigma(I_{i+1})$.
Otherwise, we say that the QSS is \textbf{interlacing}. 

\

Note that, unlike the case on QSS, reordering an interlacing QSS does not in general give an interlacing QSS.
\end{defin}

\begin{thm}[Antipode formula for Permutation pattern Hopf algebra]\label{thm:antipode_perms_intro}
Let $\pi = \pi_1\oplus \dots \oplus \pi_n$ be a decomposition of a permutation $\pi$ into $\oplus$-indecomposable permutations.
Then, we have the following formula for the antipode of $\pat_{\pi}$:

$$S(\pat_{\pi}) = (-1)^n \sum_{\sigma} \bigl[\!\begin{smallmatrix} \sigma \\ \pi_1, \dots, \pi_n \end{smallmatrix}\!\bigr] \pat_{\sigma}\, ,$$
where the sum runs over all permutations $\sigma$, and the coefficients count the number of \textbf{interlacing QSS} of $\sigma$ from $\pi_1, \dots, \pi_n$.
\end{thm}

As in the case of graphs, this antipode formula has consequences in polynomial invariants. Also, the notion of \emph{bialgebra in cointeraction} (c.f. \cite{Foissy}) can be used to obtain antipode formulas for the pattern Hopf algebra.
These are discussed in \cite{penaguiao2023polynomial}.

\

In the following we present some examples that help discern QSS and interlacing QSS.
%We also present an antipode formula for the packed words pattern Hopf algebra, with a similar characterization.

\

\begin{figure}[h]
    \centering
    \includegraphics{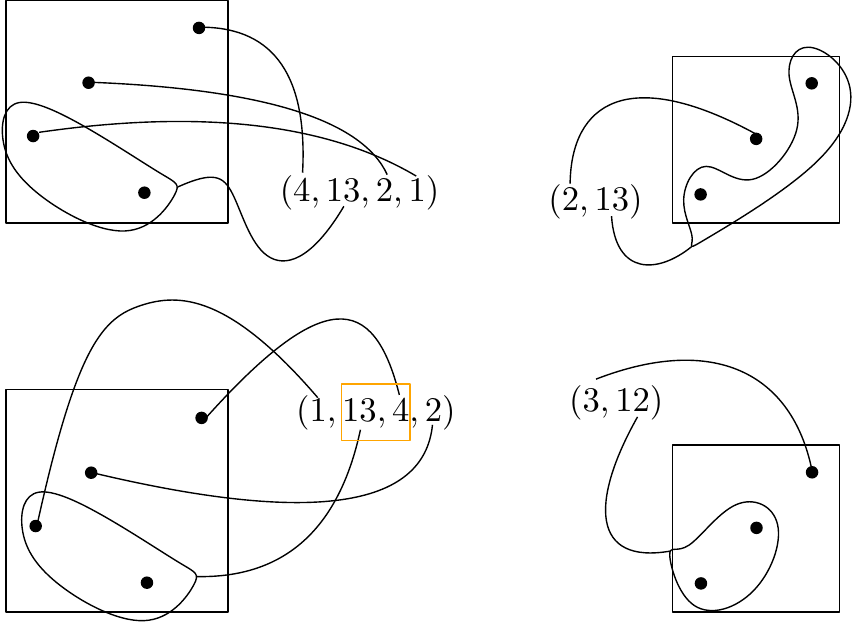}
    \caption{\textbf{Left:} the permutation 2314, along with a labelling of two of its QSS from $1, 21, 1, 1$. In orange the two sets that do not interlace. \textbf{Right:} the permutation 123, along with its two interlacing QSS from $1, 12$.\label{fig:interlacingQSSsmpl}}
\end{figure}

\begin{smpl}[Interlacing QSS]
The permutation $2314$ has several QSS from $1, 21, 1, 1$, for instance $(4, 13, 2, 1)$ and $(1, 13, 4, 2)$, but from these two, only the first is interlacing.
In \cref{fig:interlacingQSSsmpl}, one can see these two QSS.
Further computations can show that $\binom{2314}{1, 21, 1, 1} = 36$ and $\bigl[\!\begin{smallmatrix} 2314 \\ 1, 21, 1, 1 \end{smallmatrix}\!\bigr] = 8$.

\

One can observe that there are three QSS of $123$ from $1$, $12$, but one of them is non-interlacing (the QSS $(1,23)$), so
$\bigl[\!\begin{smallmatrix} 123 \\ 1, 21 \end{smallmatrix}\!\bigr] = 2$.
In \cref{fig:interlacingQSSsmpl}, one can see these two interlacing QSS.
\end{smpl}

\section{Computing the antipode\label{sec:antipode_computing}}

In this section we will now explore the \textit{ad hoc} methods to obtain the antipode, using basic axioms and the Takeuchi formula, on the polynomial algebra and on the permutation pattern Hopf algebra.
These methods fail to give a cancellation-free and grouping-free formula for the antipode.
Afterwards we employ the sign-reversing involution method to compute the antipode on the polynomial Hopf algebra.
This serves as a display of how the method works.

\subsection{\textit{Ad hoc} methods}
We start by recalling Takeuchi's formula, in the form that is presented in \cite{GrinbergReiner}, as well as some convenient notation.
%This is the beginning of any cancellation-free formula.
Let us define the $\star$ notation on maps $a, b: C \to A$, whenever $A$ is an algebra, and $C$ is a coalgebra, we define:
$$a \star b \coloneqq \mu_A \circ (a \otimes b) \circ \Delta_C\, ,$$
which defines an associative and unitary product on linear maps from $C$ to $A$. We will be focused on the case when $C=A=H$ is a Hopf algebra, so this defines a convolution operation on $\End(H)$.

\begin{prop}[Takeuchi's formula, Lemma 14]\label{lm:takeuchi}
If $H = (H, \mu, \iota, \Delta, \epsilon, S)$ is a Hopf algebra such that $(\iota\circ \epsilon - \id_H)$ is $\star$-nilpotent, then 
\begin{equation}\label{eq:eq1}
S = \sum_{k\geq 0 }  ( \iota  \circ\epsilon- \id_H)^{\star k} = \sum_{k\geq 0} (-1)^k \mu^{\circ (k-1)} \circ (\id_{H} - \iota \circ \epsilon)^{\otimes k} \circ \Delta^{\circ (k-1)}\, .
\end{equation}

We use the convention that $\Delta^{\circ (-1)} = \epsilon $ and $\mu^{\circ (-1)} = \iota$.
\end{prop}

Note that, for any filtered Hopf algebra, $( \iota \circ \epsilon - \id_H)$ is $\star$-nilpotent.
Therefore, for any pattern algebra, Takeuchi's formula holds.
%This is summarized in \cref{thm:conHopfalgebra}.

\

In the Hopf algebra of polynomials, this gives us the following:
$$S(x^3) = \underbrace{0}_{k = 0} - \underbrace{x^3}_{k = 1} + \underbrace{3 x^2 \cdot x + 3 x \cdot x^2}_{k=2} - \underbrace{6 x \cdot x \cdot x}_{k = 3} = - x^3 \, .$$

We now present another example, this time on the permutation pattern Hopf algebra $\mathcal A(\mathtt{Per})$.
Consider $\pi = 132 = 1 \oplus 21$. Then Takeuchi's formula gives us:
\begin{align*}
S(\pat_{132}) =& \sum_{k=0}^2 (-1)^k \mu^{\circ k-1} \circ (\id_{\mathcal A(\mathtt{Per})} - \iota \circ \epsilon)^{\otimes k} \circ \Delta^{\circ k-1}(\pat_{132})\\
=& -(\id_{\mathbb{K}[x]} - \iota \circ \epsilon)(\pat_{132}) + \mu \circ (\id_{\mathbb{K}[x]} - \iota\circ\epsilon)^{\otimes 2}(\pat_1 \otimes \pat_{12}) \\
=& - \underbrace{\pat_{132}}_{k=1} + \underbrace{\pat_1 \pat_{21}}_{k=2} \\
=& 3 \pat_{321} + 2 \pat_{231} + 2 \pat_{312} + \pat_{213} + 2 \pat_{21} \, .
\end{align*}

\

These coefficients can be seen as enumerating quasi-shuffle signatures of $132$ from $1$ and $12$ that are \textbf{interlacing}, according to \cref{thm:antipode_perms_intro}.

\

\subsection{The sign-reversing involution method}

The application of the sign-reversing involution method to compute antipodes of Hopf algebras was first presented in \cite{BS2017}.
This is a method to find cancellation-free formulas for the antipode of a Hopf algebra.
It starts in the formula given by Takeuchi, and keeps track of all the terms to be summed, usually by means of compositions, that arise in this formula.
Thus, the sum obtained runs over a collection of objects, say $\mathcal O$, that is partitioned into families indexed by compositions.
These compositions play an important role in the sum, as their length determines the sign of the corresponding objects.

\

Recall that an involution $\zeta $ is an endomorphism such that $\zeta \circ \zeta$ is the identity.
We describe an involution in $\mathcal O$, call it $\zeta $, in such a way that if $\zeta(x) \neq x$, then $x$ and $\zeta(x)$ contribute with opposite signs to the antipode.
As a consequence, when applying Takeuchi's formula, we can cancel terms that are not fixed points of $\zeta$.
We give an example of how this method is applied in the Hopf algebra $\mathbb{K}[x]$.

\

The following is a computation from \cite{BS2017}.
We would like to point out that there are easier ways to obtain an antipode formula for the polynomial Hopf algebra: the interested reader can find some for instance in \cite{GrinbergReiner}.
However, it shows the power of Takeuchi’s formula, as the whole process can be done with little recourse to intuition.

\begin{thm}[The antipode formula for the polynomial Hopf algebra]\label{thm:polyHA}
The antipode $S$ for $\mathbb{K}[x] $ is 
$$ S(x^n) =(-x)^n\, . $$
\end{thm}

To prove this formula, we first introduce \textit{weak compositions}.
A \textbf{weak composition} $\alpha$ of an integer $n$ is a list of non-negative integers $\alpha = (\alpha_1, \dots , \alpha_l)$ such that $\sum_i \alpha_i = n$.
We denote the length of the weak composition by $l = l(\alpha)$, and use the shorthand notation $\alpha \models^0 n$.
If $\alpha $ has no zero entries, we say that $\alpha$ is a \textbf{composition}, and write $\alpha \in \mathcal C_n$.

\

A \textbf{weak set composition} $\opi$ of a set $A$ is a list of pairwise disjoint sets $\opi = (A_1, \dots , A_l)$ such that $\bigcup_i A_i = A$.
Note that some sets may be empty.
We denote the length of the weak composition by $l= l(\opi)$, and use the shorthand notation $\opi \models^0 A$ to indicate that $\opi$ is a weak set composition of the set $A$.
If the weak set composition does not contain any empty set, we say that this is a \textbf{set composition}, and we denote $\opi \models A$ to indicate that $\opi$ is a set composition of the set $A$.
There are finitely many set compositions of a given set $A$, whereas there are infinitely many weak set compositions of $A$.
Let $\mathbf{C}_A$ be the collection of set composition of $A$.
We abbreviate to $\mathbf{C}_n$ when $A = [n]$.

\begin{proof}[Proof of \cref{thm:polyHA}]
We use Takeuchi formula, given in \eqref{eq:eq1}, which holds for graded Hopf algebras,
$$S(x^n) = \sum_{k = 0} (-1)^k \mu^{\circ (k-1)} \circ (\id_{\mathbb K[x]} - \iota \circ \epsilon )^{\otimes k} \circ \Delta^{\circ (k-1)} (x^n) \, .$$

The key is use:
$$ \Delta^{\circ (k-1) } (x^n) = \sum_{(A_1, \dots , A_k) \models^0 [n] } x^{|A_1|} \otimes \dots \otimes x^{|A_k|} \, ,$$
where the sum runs over weak set compositions of the set $[n]$ with lenght $k$.
This can be shown with easy induction on $n$.
Thus, the antipode formula can be rewritten as

\begin{align*}
S(x^n)&= \sum_{k=0}^n(-1)^k \sum_{(A_1, \dots , A_k) \models^0 [n]} \mu^{\circ (k-1)} \circ (\id_{\mathbb K[x]} - \iota \circ \epsilon )^{\otimes k} (x^{|A_1|} \otimes \dots \otimes x^{|A_k|})   \\
	  &= \sum_{k=0}^n(-1)^k \sum_{(A_1, \dots , A_k) \models [n]} \mu^{\circ (k-1)} (x^{|A_1|} \otimes \dots \otimes x^{|A_k|})   \\
	  &= \sum_{k=0}^n(-1)^k \sum_{(A_1, \dots , A_k) \models [n]} x^n   \\
	  &= x^n \sum_{\opi \models [n]} (-1)^{l(\opi)}\, .
\end{align*}

Consider the following involution $\zeta: \oPi_{n} \to \oPi_{n} $.
For $\opi = (A_1, \dots , A_k) $, let $j_{\opi}$ be the smallest index such that $|A_{j_{\opi}}| \neq 1$ or $\max A_{j_{\opi}} > \max A_{j_{\opi}+1} $.
Then, there are three cases:

\begin{enumerate}

\item The set $A_{j_{\opi}} $ is a singleton with $j_{\opi}\leq k-1$, then let $\zeta(\opi ) $ be the set composition resulting from merging $A_{j_{\opi}} $ and $A_{j_{\opi}+1}$.
Note how, in this case, $\max ( A_{j_{\opi}} \cup A_{j_{\opi} + 1} ) $ is the only element in $A_{j_{\opi}}$.

\item The set $A_{j_{\opi}} $ is not a singleton, then we define $\zeta(\opi ) $ to be the set composition resulting from splitting $A_{j_{\opi}} $ into $\{ \max A_{j_{\opi}} \} $ and $A_{j_{\opi}} \setminus \{\max A_{j_{\opi}} \}$, in this order.

\item There is no such $j_{\opi}$. Then $\opi = (\{1\}, \dots , \{n\})$ and we define $\zeta(\opi )= \opi $.

\end{enumerate}

It is a direct observation that $\zeta $ is an involution.
In fact, the only fixed point is $\opi = (\{1\}, \dots , \{n\})$, and for any other set composition $\opi$, whenever the index $j_{\opi}$ in $\opi $ behaves as described in case 1, then the index $j_{\zeta(\opi)}$ in $\zeta ( \opi ) $ behaves as described in case 2, in which case we have $l(\opi) = 1+l(\zeta (\opi))$ and we can easily see that $\zeta(\zeta(\opi )) = \opi$.
%We also have the converse statement. 
Thus, we have that 
\[x^n \sum_{\opi \in \oPi_n} (-1)^{l(\opi)} = x^n (-1)^{l(\{1\}, \dots, \{n\} )}= (-x)^n,\] 
as desired.
\end{proof}

In this way we see that a formula for the antipode depends simply on an understanding of the structure of the length of compositions of $[n]$.
This is a general feature whenever we apply Takeuchi's formula.

\

\section{Combinatorial species\label{sec:species}}

In this section we will give the preliminaries of monoids in species.
This will follow closely \cite{AM2010} and \cite{Schmitt1993}.
Specifically, we will introduce species with vector spaces and sets.
We will also introduce species with restrictions, and we will clarify the meaning of a monoid, comonoid, bimonoid and Hopf monoid in each of these monoidal categories.
We will finally present some examples of species with restrictions that will be important in the remaining paper. 

\subsection{Species}
In this section we recall the basic definitions of the general theory of \emph{combinatorial species}. Following \cite{AM2010}, we will focus first in \emph{vectorial species} and \emph{set species}.

\

Let $\mathbb{K}$ be a field of arbitrary characteristic. Let $\Fset$ be the category of finite sets and bijections between finite sets, and $\Vect_{\mathbb{K}}$ be the category of $\mathbb{K}$-vector spaces and linear maps between vector spaces. A {\bf vector species}, or simply a \textbf{species}, is a functor $\tp: \Fset \to \Vect_{\mathbb{K}}$. A morphism between species $\tp$ and $\tq$ is a natural transformation between the functors $\tp$ and $\tq$.
We will always denote the vector species with a bold lowercase Latin letter, with few exceptions.

\

A species $\tp$ is said {\bf positive} is $\tp[\emptyset]=0$. The {\bf positive part} of a species $\tq$ is the positive species $\tq_+$ given by
\[\tq_+[I]=\begin{cases}
\tq[I],& \text{ if } I\neq \emptyset\\
\emptyset, & \text{ otherwise}
\end{cases}.\]

Given a vector space $V$, let ${\bf 1}_V$ be the vector species defined by
\[{\bf 1}_V[I]=\begin{cases}
V,& \text{ if } I= \emptyset\\
\emptyset, & \text{ otherwise}
\end{cases}.\]

\

We write $\Ss_{\mathbb{K}}$ for the category of vector species over the field $\mathbb{K}$. There are several possible monoidal structures on this category. 
We will consider two of them: the {\bf Cauchy} and {\bf substitution} products $\cdot$ and $\circ$, respectively: for any finite set $I$,
\[(\tp \cdot \tq)[I]:=\bigoplus_{I = S \sqcup T}\tp[S]\otimes \tq[T];\]
\[(\tp \circ \tq)[I]:=\bigoplus_{X \vdash I}\tp[X]\otimes \left(\bigotimes_{B \in X} \tq[S] \right).\]

We denote by $(\Ssk, \cdot)$ and $(\Ssk, \circ)$ the resulting monoidal categories obtained from the Cauchy and substitution operations, respectively.

\

We can also consider {\bf set species}, whose are simply functors $\rP: \Fset \to \Set$, where $\Set$ is the category of arbitrary sets and arbitrary maps between sets. Given a set species $\rP$, the notions of {\bf positive part} $\rP_+$ of $\rP$, {\bf positive set species} are defined analogously as for vector species. If $C$ is a set, let ${\bf 1}_C$ be the set species defined by
\[{\bf 1}_C[I]=\begin{cases}
C,& \text{ if } I= \emptyset\\
\emptyset, & \text{ otherwise}
\end{cases},\]
for any finite set $I$.
We will always denote a set species with a capital Latin letter, with few exceptions.

\

The Cauchy and substitution products of vector species have their analogues in this context. For instance, if $\rP$ and $\rQ$ are two set species, let
\[(\rP \cdot \rQ)[I]:=\bigsqcup_{I = S \sqcup T}\rP[S]\times \rQ[T];\]
\[(\rP \circ \rQ)[I]:=\bigsqcup_{X \vdash I}\rP[X]\times \left(\prod_{B \in X} \rQ[S] \right),\]
on any finite set $I$, where the $\times$ symbol in the right-hand sides refers to the Cartesian product.
We write $(\Ss, \cdot)$ for the monoidal category of set species.

\

It is possible to relate set species to vector species via the \emph{linearization functor} $\mathbb{K}(-): \Set \to \Vect_{\mathbb{K}}$, which sends a set to the vector space generated by the given set. Composing a set species $\rP$ with the linearization functor gives a vector species, denoted by $\mathbb{K}\rP$. A {\bf linearized species} is a vector species $\tp$ of the form $\tp=\mathbb{K}\rP$, for some set species $\rP$. We have natural isomorphisms
 \[\mathbb{K}(\rP \cdot \rQ)\simeq \mathbb{K}\rP \cdot \mathbb{K}\rQ \qquad , \qquad \mathbb{K}(\rP \circ \rQ)\simeq \mathbb{K}\rP \circ \mathbb{K}\rQ.\]

\

\subsection{Algebraic structures on vector species}

\subsubsection{Monoids}
A {\bf monoid} in $(\Ssk, \cdot)$  consist of a species $\tp$ equipped with morphisms of species
\begin{equation*}
    \mu: \tp \cdot \tp \to \tp \qquad \text{ and } \qquad \iota: \mathrm{1}_{\mathbb{K}} \to \tp.
\end{equation*}

That is, for each finite set $I$ and for each decomposition $I=S \sqcup T$, we have a linear map 
\begin{equation*}
    \mu_{S,T}: \tp[S] \otimes \tp[T]\to \tp[I] \text{ and } \iota_\emptyset: \mathbb{K} \to \tp[\emptyset].
\end{equation*}

If $x \in \tp[S]$, $y \in \tp[T]$,  let 
\[x \cdot y \in \tp[I]\]
denote the image of $x\otimes y$ under $\mu_{S,T}$. 

\

The collection of linear maps $\mu=(\mu_{S,T})$, called the {\bf product} of the monoid, must satisfy the following axioms.

\

\begin{itemize}
    \item[(i)] Naturality axiom: for finite sets $I,J$, a bijection $\sigma: I \to J$, a decomposition $I=S \sqcup T$ and elements $x \in \tp[S]$ and $y \in \tp[T]$, we have
\begin{equation*}
\tp[\sigma](x \cdot y)=\tp[\sigma(S)](x) \cdot \tp[\sigma(T)](y) \, .
\end{equation*}

\item[(ii)] Associativity axiom: for finite set $I$, a decomposition $I=R \sqcup S \sqcup T$ and for elements $x \in \tp[R]$, $y \in \tp[S]$ and $z \in \tp[T]$, we have
\begin{equation}\label{eq:axiomii}
    (x \cdot y)\cdot z=x \cdot (y \cdot z).
\end{equation}

\item[(iii)] Unit axiom: for each finite set $I$ and $x \in \tp[I]$, we have
\begin{equation*}
 x \cdot \iota_\emptyset(1) = x =  \iota_\emptyset(1) \cdot x,
\end{equation*}
where $1\in \mathbb{K}$ is the unit of the filed $\mathbb{K}$.

\end{itemize}

\

A monoid $(\tp, \mu, \iota)$ in $(\Ssk, \cdot)$ is {\bf commutative} if
\begin{equation*}
    x\cdot y=y\cdot x,
\end{equation*}
for all $I=S \sqcup T$, $x \in \tp[S]$ and $y \in \tp[T]$.

\

Let $(\tp, \mu, \iota)$ be a monoid. From the associativity axiom, for any decomposition $I=S_1 \sqcup \cdots \sqcup S_k$ with $k \geq 2$ there is a unique map
\begin{equation}
    \mu_{S_1, \hdots, S_k}: \tp[S_1]\otimes \cdots \otimes \tp[S_k] \to \tp[I],
\end{equation}
called the {\bf higher product map} of $\tp$, obtained by iterating the product maps in any meaningful way. 
This is well defined from \eqref{eq:axiomii}.
We can extend the definition of higher product map for all $k\geq 0$: for $k=1$, $\mu_I$ is defined as the identity map of $\tp[I]$ ($I$ is the only decomposition of itself in one block); if $k=0$, then $\mu_{\empty}:=\iota_{\empty}$.
Note how in this case $I = \emptyset $.
\

Monoids are closed under the Cauchy product. Also, if $(\tp, \mu, \iota)$ is a monoid, then $\tp[\emptyset]$ is an algebra with product $\mu_{\emptyset, \emptyset}$ and unit $\iota_\emptyset(1)$ (see \cite{AM2013}, section 2.3).

\

\subsubsection{Comonoids}
A {\bf comonoid} in $(\Ssk, \cdot)$ corresponds to the dual notion of monoids in $(\Ssk, \cdot)$. Formally, a comonoid consists of a species $\tp$ equipped with morphisms of species
\begin{equation*}
    \Delta: \tp \to \tp \cdot \tp \qquad \text{ and } \qquad \varepsilon: \tp \to 1_{\mathbb{K}}.
\end{equation*}

That is, for each finite set $I$ and for each decomposition $I=S \sqcup T$, we have a linear map 
\begin{equation*}
    \Delta_{S,T}: \tp[I] \to \tp[S] \otimes \tp[T] \text{ and } \varepsilon_\emptyset: \tp[\emptyset] \to \mathbb{K}.
\end{equation*}

\

The collection of linear maps $\Delta=(\Delta_{S,T})$, called the {\bf coproduct} of the monoid, must satisfies the following axioms.

\

\begin{itemize}
    \item[(i)] Naturality axiom: for finite sets $I, J$, bijection $\sigma: I \to J$, a decomposition $I = S \sqcup T$ and an element $x\in \tp [I]$.
\begin{equation*}
(\tp[\sigma|_S] \otimes \tp[\sigma|_T])\circ\Delta_{S,T}(x)=\Delta_{\sigma(S), \sigma(T)} \circ\tp[\sigma](x).
\end{equation*}

\item[(ii)] Coassociativity axiom: for a finite set $I$, a decomposition $I=R \sqcup S \sqcup T$ and for each $x \in \tp[I]$, we have
\begin{equation}\label{eq:coaxiomii}
    (\Delta_{R,S}\otimes \text{id}_{\tp[T]})\circ \Delta_{R \sqcup S, T}(x)=(\text{id}_{\tp[R]} \otimes \Delta_{S,T})\circ \Delta_{R, S \sqcup T}(x).
\end{equation}
\vspace{.1in}
\item[(iii)] Counit axiom: for each finite set $I$ and $x \in \tp[I]$, we have
\begin{equation*}
(\varepsilon_\emptyset \otimes \text{id}_{\tp[I]})\circ\Delta_{\emptyset, I}(x)= x = (\text{id}_{\tp[I]} \otimes \varepsilon_\emptyset)\circ\Delta_{I,\emptyset}(x).
\end{equation*}

\end{itemize}

\

A comonoid $(\tp, \mu, \iota)$ in $(\Ssk, \cdot)$ is {\bf cocommutative} if for any finite disjoint sets $S, T$, and element $x\in \tp[S\sqcup T]$, we have
\begin{equation*}
    \Delta_{S,T}(x)=\Delta_{T,S}(x).
\end{equation*}

\

Let $(\tp, \mu, \iota)$ be a comonoid. Dually to the case of monoids, every decomposition $I=S_1 \sqcup \cdots \sqcup S_k$ with $k \geq 0$ gives rises to a unique linear map
\begin{equation}
    \Delta_{S_1, \hdots, S_k}: \tp[I] \to \tp[S_1]\otimes \cdots \otimes \tp[S_k],
\end{equation}
called the {\bf higher coproduct map} of $\tp$, obtained by iterating the coproducts map $\Delta_{S,T}$. 
This is well defined because of \cref{eq:coaxiomii}.
çwe extend this definition to $k=1$ as the identity of $\tp[I]$ and for $k=0$, this map is the counit map $\varepsilon_\emptyset$.

\

Comonoids are closed under the Cauchy product. 
If $(\tp, \Delta, \varepsilon)$ is a comonoid, then $\tp[\emptyset]$ is a coalgebra with coproduct $\Delta_{\emptyset, \emptyset}$ and counit $\varepsilon_\emptyset$ (see \cite{AM2013}, section 2.4).

\

\subsubsection{Bimonoids  and Hopf monoids}
A {\bf bimonoid} $(\thh, \mu, \Delta, \iota, \varepsilon)$ in $(\Ssk, \cdot)$ is a monoid and comonoid such that the diagram
\[\xymatrix{
\thh[A] \otimes \thh[B] \otimes \thh[C] \otimes \thh[D] \ar[rr]^-{\cong} && \thh[A] \otimes \thh[C] \otimes \thh[B] \otimes \thh[D] \ar[dd]^-{\mu_{A,C}\otimes\mu_{B,D}}\\
&&\\
\thh[S_1]\otimes \thh[S_2]\ar[uu]^-{\Delta_{A,B}\otimes \Delta_{C,D}} \ar[r]_-{\mu_{S_1, S_2}} & \thh[I] \ar[r]_-{\Delta_{T_1, T_2}} & \thh[T_1]\otimes \thh[T_2]
}\]

commutes, where $I=S_1\sqcup S_2=T_1 \sqcup T_2$ are two decompositions of a finite set $I$ with the following resulting pairwise intersections:
\[A:=S_1\cap T_1 \quad , \quad B:=S_1 \cap T_2 \quad , \quad C:=S_2 \cap T_1 \quad , \quad D:=S_2 \cap T_2.\]
This is also schematically presented in \cref{fig:bimonoid}.

\begin{figure}
\includegraphics[scale=0.5]{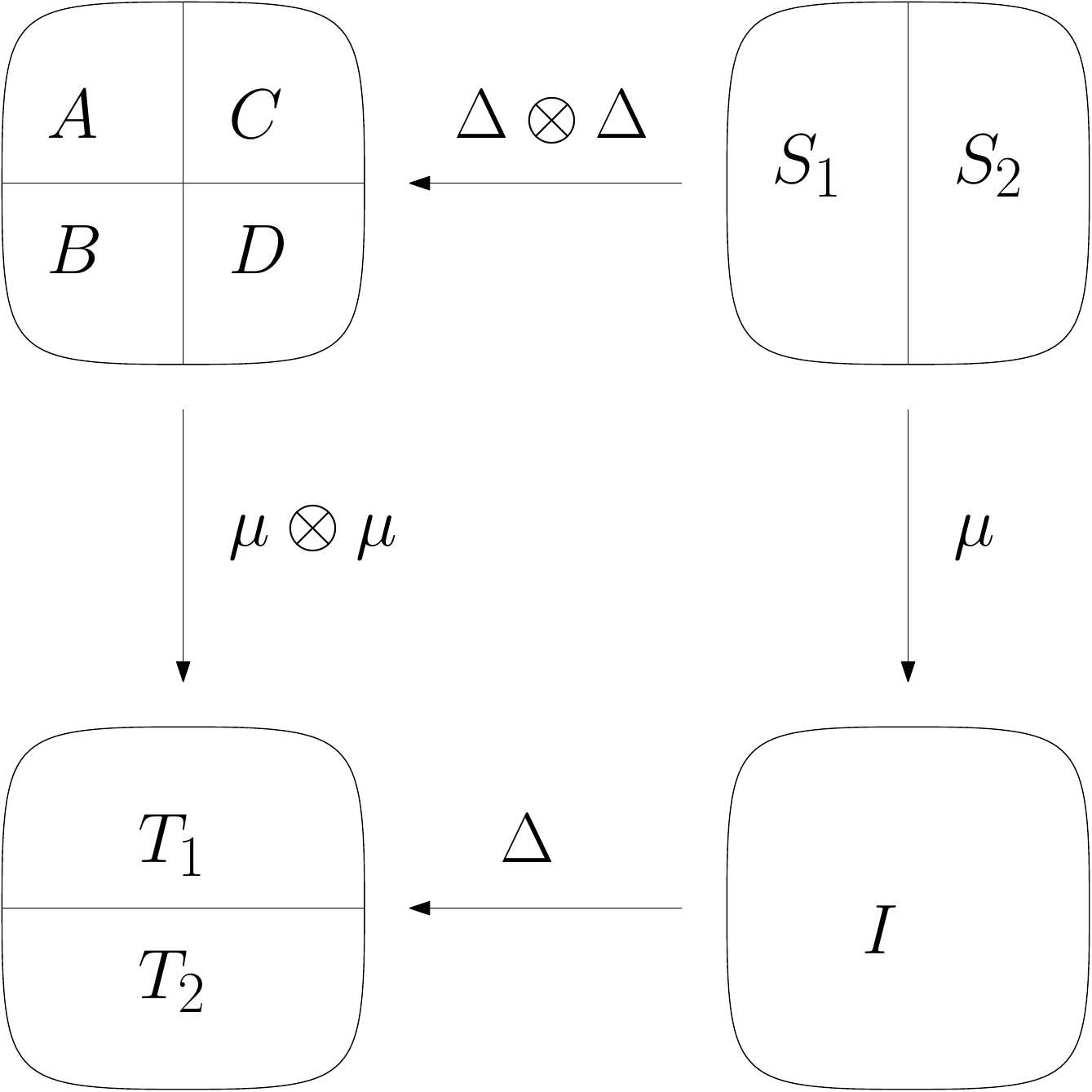}
\caption{The bimonoid compatibility axiom.\label{fig:bimonoid}}
\end{figure}

\

We can define the convolution algebra $\text{End}_{\Ssk}(\thh)$ as the set of natural transformations $l: \thh \to \thh$ with the product $\star$.

A morphism of species $s: \thh \to \thh$, is callen and {\bf antipode} of $\thh$, if $\thh[\emptyset]$ is a Hopf algebra with antipode $s_\emptyset: \thh[\emptyset] \to \thh[\emptyset]$, and for each nonempty set $I$, we have
\begin{equation}
    \sum_{S \sqcup T = I} \mu_{S,T}(\text{id}_S \otimes s_T)\Delta_{S,T} = 0 = \sum_{S \sqcup T = I} \mu_{S,T}(s_S \otimes \text{id}_T)\Delta_{S,T}.
\end{equation}

A {\bf Hopf monoid} in $(\Ssk, \cdot)$ is a bimonoid along with an antipode $s:\thh \to \thh $.
Equivalently, the identity map $\text{id}_{\thh}$ is invertible in $\text{End}_{\Ssk}(\thh)$ (c.f. \cite{AM2010}, section 2.7).
Recall that this is also the case in the classical Hopf algebras.
\

\subsection{Algebraic structures on set species}

The notions of monoid, comonoid, bimonoid and Hopf monoid for set species can be described in terms similar to those in the previous section. 

\

\subsubsection{Monoids}
A {\bf monoid in set species} consist of a species $\rP$ equipped with morphisms of species
\begin{equation}
    \mu: \rP \cdot \rP \to \rP \qquad \text{ and } \qquad \iota: \mathrm{1}_{\{\emptyset \} } \to \rP.
\end{equation}

That is, for each decomposition $I=S \sqcup T$, we have maps 
\begin{equation}
    \mu_{S,T}: \rP[S] \times \rP[T]\to \rP[I] \text{ and } \iota_\emptyset: \{\emptyset\} \to \rP[\emptyset].
\end{equation}

If $x \in \rP[S]$, $y \in \rP[T]$,  let 
\[x \cdot y \in \rP[I]\]
denote the image of $(x,y)$ under $\mu_{S,T}$. Also, let $e\in \rP[\emptyset]$ denote the image of $\emptyset$ under $\iota_\emptyset$.

\

The collection of maps $\mu=(\mu_{S,T})$, called the {\bf product} of the monoid, must satisfies naturality, associativity and unit axioms analogue to the ones defined for monoids in vector species.

\

Note that, for any monoid on a set species $(\rP, \mu, \iota)$, then $(\rP[\emptyset], \mu_{\emptyset, \emptyset}, \iota_\emptyset)$ is a set theoretical monoid.

\

A monoid in set species $(\rP, \mu, \iota)$ is {\bf commutative} if
\[x\cdot y=y\cdot x,\]
for all $I=S \sqcup T$, $x \in \rP[S]$ and $y \in \rP[T]$.

\

\subsubsection{Comonoids}
A {\bf comonoid in set species} consist of a species $\rP$ equipped with morphisms of species
\[\Delta:\rP \to \rP \cdot \rP \qquad \text{ and } \qquad \varepsilon: \rP \to \mathrm{1}_{\{\emptyset \} }. \]
That is, for each decomposition $I=S \sqcup T$ we have maps $\Delta_{S,T}: \rP[I]\to \rP[S] \times \rP[T]$, and $\varepsilon_\emptyset: \rP[\emptyset] \to \{\emptyset\}$. If $x \in \rP[I]$,  let 
\[(x|_S, x /_S)\in \rP[S]\times \rP[T]\]
denote the image of $(x,y)$ under $\Delta_{S,T}$. The map $x \mapsto x|_S$ can be thought as a ``restriction'' of the structure $x$ from $I$ to $S$, while $x \mapsto x/_S$ can be associated to a ``contraction'' of $S$ from $x$, resulting in a structure on $T$.

\

The collection of maps $\Delta=(\Delta_{S,T})$, called the {\bf coproduct} of the monoid, must satisfy naturality, coassociativity and counit axioms analogues to the ones defined for monoids in vector species. In terms of the restriction/contraction notation, these axioms are described as follows:
\begin{itemize}
    \item Naturality axiom: for each bijection $\sigma: I \to J$,  we have
\[{\Big (}\rP[\sigma](x){\Big )}|_{\sigma(S)}=\rP[\sigma|_S](x|_S) \qquad , \qquad {\Big (}\rP[\sigma](x){\Big )}/_{\sigma(S)}=\rP[\sigma|_T](x/_S),\]
for all $x \in \rP[I]$.
\vspace{.1in}
\item Coassociativity axiom: for all decomposition $I=R \sqcup S \sqcup T$, \[(x|_{R\sqcup S})|_R=x|_R \qquad , \qquad (x|_{R\sqcup S})/_R=(x/_R)|_S \qquad , \qquad x/_{R \sqcup S}=(x/_R)/_S, \]
for all $x \in \rP[I]$
\vspace{.1in}
\item Counit axiom: we have
\[x|_I= x = x/_\emptyset,\]
for each finite set $I$ and for each $x \in \rP[I]$. In particular, $\Delta_{\emptyset, \emptyset}(x)=(x,x)$, for each $x \in \rP[\emptyset]$.
\end{itemize}

\

A comonoid in set species $(\rP, \Delta, \varepsilon)$ is {\bf cocommutative} if
\[x|_S=x/_T,\]

for any disjoint finite sets $S, T$ and $x \in \rP[S\sqcup T]$.

\

\subsubsection{Bimonoids}
A bimonoid in set species $(\rH, \mu, \Delta, \iota, \varepsilon)$ is a monoid and comonoid in set species such that the diagram
\[\xymatrix{
\rH[A] \times \rH[B] \times\rH[C] \times \rH[D] \ar[rr]^-{\cong} && \rH[A] \times \rH[C] \times \rH[B] \times \rH[D] \ar[dd]^-{\mu_{A,C}\times\mu_{B,D}}\\
&&\\
\rH[S_1]\times \rH[S_2]\ar[uu]^-{\Delta_{A,B}\times \Delta_{C,D}} \ar[r]_-{\mu_{S_1, S_2}} & \rH[I] \ar[r]_-{\Delta_{T_1, T_2}} & \rH[T_1]\times \rH[T_2]
}\]
commutes, where $I=S_1\sqcup S_2=T_1 \sqcup T_2$ are two decompositions of a finite set $I$ with the following resulting pairwise intersections:
\[A:=S_1\cap T_1 \quad , \quad B:=S_1 \cap T_2 \quad , \quad C:=S_2 \cap T_1 \quad , \quad D:=S_2 \cap T_2.\]

The compatibility axiom in the definition of bimonoid can be reformulated as
\[x|_A \cdot y|_C= (x\cdot y)|_{T_1} \qquad , \qquad x/_A \cdot y/_C=(x \cdot y)/_{T_1},\]
for any disjoint sets $S_1, S_2$, for elements $x \in \rH[S_1]$, $y\in \rH[S_2]$ and for any set $T_1\subseteq S_1 \sqcup S_2$, by letting $A= S_1 \cap T_1$ and $C= S_2\cap T_1$.

\

A {\bf Hopf monoid} in $(\Ss, \cdot)$ $\rH$ is a bimonoid in set species such that the monoid $\rH[\emptyset]$ is a group. Its antipode is the map $s_\emptyset: \rH[\emptyset]\to \rH[\emptyset]$ given by $s_\emptyset(x):=x^{-1}$.
This allows us to define an antipode map $s: H\to H$ via the Takeuchi formula, adapted to set species.

\

\subsection{Fock functor}
In \cite{AM2010} (Part III), a construction is presented allowing to produce a (graded) Hopf algebra from a Hopf monoid. This is a categorical approach of a construction due to Stover (\cite{Stover}, Section 14), studied later by Patras, Schocker and Reutenauer in \cite{PR2004}, \cite{PS2006} and \cite{PS2008}.

\

We recall briefly this construction. Let $\mathbb{K}$ be a field of characteristic zero. If $\tp \in \Ssk$, then there is an action of the symmetric group $\mathfrak{S}_n$ on $\tp[n]$ by relabeling, for each $n \geq 0$. We denote by $\tp[n]_{\mathfrak{S}_n}$ the \emph{space of $\mathfrak{S}_n$-coinvariants of $\tp[n]$}:
\begin{equation}
  \tp[n]_{\mathfrak{S}_n}:=\tp[n]/\left\langle \, x- \tp[\alpha](x) \, | \, \alpha \in \mathfrak{S}_n, x\in \tp[n] \,\right\rangle.
\end{equation}

\

Consider $\gVec$ be the category of graded vector spaces over $\mathbb{K}$. The functors $\Kc, \Kcb: \Ssk \to \gVec$ given by 
\begin{equation}
    \Kc(\tp):=\bigoplus_{n \geq 0}\tp[n] \qquad , \qquad \Kcb(\tp):= \bigoplus_{n \geq 0}\tp[n]_{\mathfrak{S}_n}
\end{equation}
are referred in \cite{AM2010} as \emph{full Fock funtor} and \emph{bosonic Fock functor}, respectively. From any monoid (resp. comonoid, Hopf monoid) $\tp$, it is possible to obtain algebras (resp. coalgebras, Hopf algebras) $\Kc(\tp)$ and $\Kcb(\tp)$ from those of $\tp$, together with certain canonical transformations (see \cite{AM2010}, section 15.2).

\

\section{The pattern Hopf algebra \label{sec:pattern_algebra_contruction}}

\subsection{Species with restriction}
The general setting for our approach to patterns is given by the notion of \emph{species with restrictions}, a terminology due to Schmitt (see \cite{Schmitt1993}) and used by the first author in \cite{Penaguiao2020}, where these were called combinatorial presheaves.

\

Let $\Fset_{\!\!\!\!\!\hookrightarrow}$ be the category of finite sets with injections as morphisms. A (set) {\bf species with restriction} is a contravariant functor $\prR:\Fset_{\!\!\!\!\!\hookrightarrow} \to \Set$. Given a species with restrictions $\prR$ and a couple of finite sets $I,J$ such that $J \subseteq I$, the {\bf restriction map} $\text{res}_{I,J}:=\text{res}[\hookrightarrow]$ is the image under the functor $\prR$ of the inclusion $J \hookrightarrow I$:
\[\text{res}_{I,J}: \prR[I]\to \prR[J].\]

By functoriality, these maps satisfy the contravariant axioms
\begin{equation}\label{Axrestr}
    \text{res}_{J,K}\circ\text{res}_{I,J}=\text{res}_{I,K} \qquad , \qquad \text{res}_{I,I}=\text{id}_{\prR[I]},
\end{equation}

for any finite sets $I \supseteq J \supseteq K$. 
Since any arbitrary injection equals a bijection followed by an inclusion, any species with restriction is equivalent to a set species together with restriction maps satisfying the axioms \eqref{Axrestr}.

\

Species with restrictions form a category $\Spr$, where the arrows are natural transformations between functors.
We denote species with restrictions with a typewriter typescript.

Notice that, for any finite set $C$, the set species $\mathrm{1}_C$ is also a species with restrictions, where $\text{res}_{\emptyset, \emptyset} = \id_C$.
\

\subsection{Schmitt's comonoid}
In \cite{Schmitt1993} (Section 3), Schmitt gave a construction of coalgebras and bialgebras from certain species. We will describes the coalgebra construction, following the notation of \cite{AM2010} (Section 8.7).

\

Given a species with restrictions $\prR$, we can constructs a linearized comonoid in $(\Ss, \cdot)$ as follows. Let $\trr=\mathbb{K}\prR$ be the linearization of $\prR$. Given a decomposition $I=S \sqcup T$, consider the linear map
\[
\Delta_{S,T}: \trr[I]\to \trr[S] \otimes \trr[T]
\]
given by
\begin{equation}\label{CoprodRestr}
\Delta_{S,T}(x):=\text{res}_{I,S}(x)\otimes \text{res}_{I,T}(x),
\end{equation}

for any $x \in \prR[I]$. Let $\epsilon_\emptyset: \trr[\emptyset]\to \mathbb{K}$ the linear extension of the map sending every element of $\prR[\emptyset]$ to $1$. Hence, we have the following result.

\begin{lm}[Schmitt]
The vector species $\trr$ is a linearized comonoid in $(\Ss, \cdot)$. In particular, the comonoid $\trr$ is cocommutative.
\end{lm}

Consider now a linearized comonoid $\tp=\mathbb{K}\rP$ in $(\Ss, \cdot)$. 
That is a comonoid in set species that has no notion of restriction.
In this case, the coproduct gives a pure tensor
\[\Delta_{S.T}(x)=x|_S \otimes x/_S,\]
for each $x \in \rP[I]$ and for each decomposition $I = S \sqcup T$. We may then define restriction maps on $\tp$ either by

\begin{align*}
\text{res}^{(1)}_{I,J}: \tp[I] &\to \tp[J] \qquad \qquad  \text{or}  &\text{res}^{(2)}_{I,J}: \tp[I] &\to \tp[J],\\
x&\mapsto x|_J \qquad &x&\mapsto x/_{I\setminus J}
\end{align*} 
for $x \in \tp[I]$. Each restriction map $\text{res}^{(1)}$ or $\text{res}^{(2)}$ turns $\tp$ into a species with restriction. When $\tp$ is cocommutative, then both restriction maps coincide. We have then the following characterization of species with restrictions (see \cite{AM2010}, Proposition 8.29, for other characterizations).

\begin{thm}\label{thm:swr_lcc}
There is an equivalence between the category of species with restrictions and the category of linearized cocommutative comonoids.
\end{thm}

\

\subsection{Monoids with restriction}
We see now that the restriction structures are stable for the Cauchy product.
Let $\prP, \prQ$ be two species with restrictions.
Given two finite sets $I$ and $J$ with an inclusion $J \hookrightarrow I$, consider the map $\text{res}_{I,J}$ defined as the sum of the maps running over all decompositions $I=S\sqcup T$:

\[\xymatrix{
\prP[S]\times \prQ[T] \ar[rrr]^-{\text{res}_{S, S\cap J} \times \text{res}_{T, T\cap J}}&&& \prP[S\cap J]\times \prQ[T\cap J] \subseteq (\prP \cdot \prQ)[J],
}\]

where the first and second restrictions on the arrow above are the restrictions maps corresponding to $\prP$ and $\prQ$, respectively, from the inclusions $S \cap U \hookrightarrow S$ and $T \cap U \hookrightarrow T$, respectively. 

\

Therefore, the Cauchy product endows the category of species with restrictions of a monoidal category $(\Spr, \cdot, \mathrm{1}_{\{\emptyset\}})$.

\

We describe now monoids in the monoidal category $(\Spr, \cdot, \mathrm{1}_{\{\emptyset \} })$. A monoid $(\prP, \mu)$ in species with restrictions is a species with restrictions $\prP$, equipped with a product map $\mu: \prP \cdot \prP \to \prP$, such that for each $J \subseteq I=S\sqcup T$, the diagram
\[\xymatrix{
\rP[S]\times \rP[T]\ar[d]_-{\mu_{S,T}} \ar[rrr]^-{\text{res}_{S, S\cap J} \times \text{res}_{T, T\cap J}}&&&\rP[S \cap J]\times \rP[T\cap J]\ar[d]^-{\mu_{S \cap J, S\cap J}}\\
\rP[I]\ar[rrr]_-{\text{res}_{I,J}}&&& \rP[J]
}\]
commutes.

\

Given a monoid $\prP$ in the monoidal category of species with restrictions $(\Spr, \cdot)$, let $\tp:=\mathbb{K}\prP$ be the linearization of the underlying set species of $\prP$. 
By \cref{thm:swr_lcc}, $\tp$ is a cocommutative comonoid. Since $\prP$ is a monoid, then $\tp$ is a monoid in the category of vector species. Moreover, the above diagram implies that the product and coproduct of $\prP$ are compatible, meaning that $\prP$ is a cocommutative bimonoid. 
This proves the following:

\begin{thm}\label{MonRestr1}
There is an equivalence between the category of monoids in $(\Spr, \cdot)$ and the category of linearized cocommutative bimonoids.
\end{thm}

\begin{thm}\label{MonRestr2}
If $\prP$ is a connected monoid in species with restrictions, then $\mathbb{K}\prP$ is a Hopf monoid in vector species.
\end{thm}

\

\subsection{Pattern functions and the pattern Hopf algebra}

Given a species with restrictions $\prR$ and two finite sets $I$ and $J$, two objects $a\in \prR[I], b\in \prR[J]$ are said to be {\bf isomorphic objects}, or $a\sim b$, if there is a bijection $\sigma:I\to J$ such that $\prR[\sigma](b)=a$. 

\

The collection of equivalence classes of a species with restrictions $\prR$ is denoted by \begin{equation}
\mathcal{G}(\prR) := \bigcup_{n\geq 0 } \prR[n]_{\mathfrak{S}_n}.
\end{equation}
In this way, the set $\mathcal G(\prR) $ is the collection of all the $\prR$-objects up to isomorphism. It is straightforward to show that $\mathcal{G}(\prR)$ is a basis for the vector space $\Kcb(\prR)$.

\

Recall that for every couple of finite set $I,J$ such that $J \subseteq I$, there is a restriction map 
\[\text{res}_{I,J}: \prR[I] \to \prR[J]\]
defined as the image by $\prR$ of the injection map $J \hookrightarrow I$. If $b \in \prR[I]$, we denote $b|_J$ for $\text{res}_{I,J}(b)$.

\begin{defin}[Patterns coefficients]\label{defin:patterncoeff}
Let $\prR$ be a species with restrictions. Given two finite sets $I,J$ such that $J \subseteq I$ and two objects $b\in \prR[I], a\in \prR[J]$,
we say that the subset $J \subseteq I$ is a {\bf pattern} of $a$ in $b$ if $b|_{J} \sim a$. More precisely, $J \subseteq I$ is a pattern of $a$ in $b$ if there exists a bijection $\sigma: J \to J$ such that
\[\prR[\sigma](a)=\text{res}_{I,J}(b).\]

We define the {\bf pattern coefficient} of $a$ in $b$ as
\begin{equation}
    \binom{b}{a}_{\!\prR} : = \left| \{J \subseteq I \, : \, b|_J \sim a \} \right| \, .
\end{equation}
\end{defin}

\

This definition only depends on the isomorphism classes of $a \in \prR[J]$ and $b \in \prR[I]$; see \cite{Penaguiao2020}. This motivates the following notion.

\begin{defin}[Patterns functions]\label{defin:pattern}
Let $\prR$ be a species with restrictions. Given a finite set $I$, we define the {\bf pattern function associated to} $a \in \prR[I]$ as the function
\[\pat_a: \mathcal{G}(\prR) \to \mathbb{K} \]
given by
\begin{equation}
    b \mapsto \binom{b}{a}_{\!\prR},
\end{equation}
for all $b \in \mathcal{G}(\prR)$.
\end{defin}

By definition $\pat_a \in \mathcal{F}(\mathcal{G}(\prR), \mathbb{K})$, where $\mathcal{F}(\mathcal{G}(\prR), \mathbb{K})$ denotes the set of functions from $\mathcal{G}(\prR)$ to $\mathbb{K}$. Without loss of generality, we denotes by $\pat_a$ the linear extension of the pattern function associated to $a$. Hence, we can consider $\{ \pat_a \}_{a\in \mathcal{G}(\prR)}$ as a family of linear functions from $\Kcb(\prR)$ to $\mathbb{K}$, indexed by $\mathcal{G}(\prR)$.
In \cref{sec:speciespermutation}, we see an example of a species with restrictions structure on permutations.

\

\begin{defin}[Patterns space]
If $\prR$ is a species with restriction, then the linear span of the pattern functions is denoted by
\begin{equation}
    \mathcal{A}(\prR):=\mathbb{K}\{\pat_a \, : \, a\in \mathcal{G}(\prR)\}.
\end{equation}
\end{defin}

By definition, $\mathcal{A}(\prR)$ is a linear subspace of the space of linear functions $\Kcb(\prR)^*$ from $\Kcb(\prR)$ to $\mathbb{K}$. The following was proven in \cite{Penaguiao2020}.

\begin{thm}
The subspace $\mathcal{A}(\prR)$ of $\Kcb(\prR)^*$ is closed under pointwise multiplication and has a unit.
It forms an algebra, called the {\bf pattern algebra associated to} $\prR$.
More precisely, if $a, b \in \mathcal G(\prR)$,
\begin{equation}\label{eq:prodrule}
\pat_ a   \pat_b = \sum_c \binom{c}{a, b}_{\! \prR} \pat_c \, ,
\end{equation}
where the coefficients $\binom{c}{a, b}_{\!\prR}$ are the number of ``quasi-shuffles'' of $a, b$ that result in $c$, specifically, if we take $c\in \prR[C]$ to be a representative of the equivalence class $c$, then:
$$ \binom{c}{a, b}_{\!\prR} = \left| \{(I, J) \, \text{ such that } \, \,  I \cup J = C \, ,\, \, c|_{I} \sim a, \, c|_{J} \sim b \} \right| \, .  $$
\end{thm} 

%For details on quasi-shuffles of combinatorial objects, the interested reader can see \cite{hoffman00,aguiar10,foissy16}.

Consider now a positive species with restriction $\prR$ endowed with an associative product $(\prR, \sq)$. Examples of associative operations on species with restrictions are the direct sum of permutations $\oplus$, introduced above. Recall by Theorem \eqref{MonRestr2} that 
$(\prR, \sq)$ is equivalent to the linearized cocommutative Hopf monoid $(\trr, \sq, \Delta)$, where $\trr=\mathbb{K}\prR$ and $\Delta$ is defined as in \eqref{CoprodRestr}. If $*$ denotes the pointwise product in $\Kcb(\trr^*)$, then the space of pattern function $\mathcal{A}(\prR)$ is a subalgebra of $\Kcb(\trr^*, *, \Delta_{\sq})$, where $\Delta_{\sq}$ denotes the dual map of $\sq$.

\

Under the natural identification of the function algebra $\mathcal{F}(\mathcal G (\prR), k)^{\otimes 2} $ as a subspace of $\mathcal{F}(\mathcal G (\prR) \times \mathcal G (\prR), k) $, we have
\begin{equation}\label{eq:coproddefin}
 \Delta_{\sq} \pat_a (b \otimes  c) =  \pat_a (b \,\sq\, c) \, .
\end{equation}
This is shown in \cref{thm:conHopfalgebra}. Therefore, we have the following coproduct in the pattern algebra $\mathcal A (\prR)$:
\begin{equation}\label{eq:coprodformula}
\Delta_{\sq} \pat_ a = \sum_{\substack{ b, c \, \in \, \mathcal G (\prR) \\ a = b \, \sq \, c}} \pat_b \otimes \pat_c \, .
\end{equation}
where the sum runs over coinvariants $b, c$ such that $a = b \,\sq \,c$. 
The relation \eqref{eq:coproddefin} is central in establishing that the coproduct $\Delta_{\sq} $ is compatible with the product in $\mathcal A (h)$.

\begin{thm}\label{thm:conHopfalgebra}
Let $(\prR, \sq, 1) $ be an associative species with restrictions.
Then the pattern algebra of $\prR$ together with the coproduct $\Delta_{\sq}$, and a natural choice of counit, forms a bialgebra.
If additionally $| \prR[\emptyset ] | = 1 $, the pattern algebra forms a filtered Hopf algebra.
\end{thm}

%Presheaves that satisfy $|h[\emptyset ]| = 1$ are called \textit{connected}.
%Connected algebraic structures are a classical resource in graded Hopf algebras, as in this way we can find an antipode through the so called \textit{Takeuchi formula}, introduced in \cite{takeuchi71}.

Some known Hopf algebras can be constructed as the pattern algebra of a species with restrictions.
An example is $Sym$, the Hopf algebra of \textit{symmetric functions}.
This Hopf algebra has a basis indexed by partitions, and corresponds to the pattern Hopf algebra of the species on set partitions.
%The pattern Hopf algebra corresponding to the presheaf on permutations described above was introduced by Vargas in \cite{vargas14}.
%Some other Hopf algebras constructed here, like the ones on graphs and on marked permutations below, are new, and some we conjecture are isomorphic to known Hopf algebras like the pattern Hopf algebra on set compositions, which may be simply the Hopf algebra of quasi-symmetric functions; see \cref{conj:QSym} below.

\

We end this section with a relevant theorem on functors of species with restrictions.

\begin{thm}[Theorem 3.8. in \cite{Penaguiao2020}]\label{thm:functoriality}
    If $\mathtt{f} : \prR \Rightarrow \mathtt{S}$ is a morphism of associative species with restrictions, between connected species, then $\mathcal A(\mathtt{f})$ is a Hopf algebra morphism that maps $\mathcal A(\mathtt{S}) \to \mathcal A(\mathtt{R})$.
\end{thm}

\section{Examples of species with restrictions \label{sec:species_restrictions}}

\subsection{Species on linear orders.\label{sec:specieslinearorders}}

The first relevant species with restriction to define is the species of total orders $\mathtt{L}$.
This has $\mathtt{L}[I] = \{\text{ total orders in $I$ }\}$, a set with $|I|!$ total orders.
The restriction of an order $\leq \in \mathtt{L}[I]$ to a subset $J$ is the induced order, which is a total order.
This defines a species with restrictions.

\

We also define an associative structure on $\mathtt{L}$.
If $\mathsf{p}$ is a total order in $I$ and $\mathsf{r}$ is a total order in $J$, and if $I, J$ are disjoint, then we can define $\mathsf{p} \ast\mathsf{r}$ a total order on $I \sqcup J$ as:

\begin{itemize}
\item if $a, b \in I$ such that $a \, \mathsf{p} \, b$, then $a \, \mathsf{p} \ast \mathsf{r} \, b$.

\item if $a, b \in J$ such that $a \, \mathsf{r} \, b$, then $a \, \mathsf{p} \ast \mathsf{r} \, b$.

\item if $a \in I$ and $b \in J$, then $a \, \mathsf{p} \ast \mathsf{r} \, b$.
\end{itemize}
This is invariant up to restrictions and relabellings, so this builds a connected associative species with restrictions. 

\

If $\mathsf{p} \in \mathtt{L}[I]$, we write $\mathbb{X}(\mathsf{p}) = I$.
The following fact is an immediate observation.

\begin{prop}\label{prop:linorderideals}
If $\mathbb{X}(\mathsf{p}) = I$, then $I$ is an order ideal in $\mathsf{p} \ast \mathsf{r}$.
\end{prop}

\

\subsection{Species on permutations\label{sec:speciespermutation}}

To fit the framework of species with restrictions, we use a rather unusual definition of permutations introduced in \cite{albert2020two}.
There, a permutation on a set $I$ is seen as a pair of total orders in $I$.
This relates to the usual notion of a permutation a bijection in the following way: if we order the elements of $I = \{a_1 \leq_P \dots \leq_P a_k \} = \{ b_1 \leq_V \dots \leq_V b_k \}$, then this defines a bijection via $a_i \mapsto b_i $ in $I$.
Conversely, for any bijection $f$ on $I$, there are several pairs of orders $(\leq_P, \leq_V)$ that correspond to the bijection $f$, all of which are isomorphic.
If $I \hookrightarrow J$, by restricting the orders on $I$ to orders on $J$ via the injective map $\hookrightarrow $, we obtain a restriction to a permutation on $J$.
The resulting species with restrictions structure is denoted by $\mathtt{Per}$.

\

It will be useful to represent permutations in $I$ as square diagrams labeled by $I$.
This is done in the following way: we place the elements of $I$ in an $|I| \times |I|$ grid so that the elements are placed horizontally according to the $\leq_P$ order, and vertically according to the $\leq_V$ order.
For instance, the permutation $\pi = \{1<_P2<_P3 , 2<_V1<_V3\}$ in $\{1, 2, 3\}$ can be represented as 
\begin{equation}
\begin{array}{|c|c|c|}
	\hline & & 3 \\
    \hline 1 & &  \\
    \hline & 2 & \\
    \hline 
\end{array}\, \, \, .
\end{equation}

In this way, there are $(n!)^2 $ elements in $\mathtt{Per}[n]$.
Up to relabelling, we can represent a permutation as a diagram with one dot in each column and row.
Thus, $\mathcal{G}(\mathtt{Per})$, the set of objects up to isomorphism, has $n!$ isomorphism classes of permutations of size $n$, as expected.

\

\label{defin:per}
If we consider a permutation $\pi$ on a set $I$, that is, a pair $(\leq_P, \leq_V) $ of total orders in $I$, we write $\mathbb{X}(\pi) = I$.
If $f:J \to I $ is an injective map, the preimage of each order $\leq_P, \leq_V$ is well defined and is also a total order in $J$.
This defines the permutation $\mathtt{Per}[f](\pi )$.
If $f$ is the inclusion map, this notion recovers the usual concept of a permutation pattern already present in the literature.

\

The $\ast $ operation on orders can be extended to a \textbf{diagonal sum} operation $\oplus $ on permutations.
This is usually referred as the shifted concatenation of permutations.
This endows $\mathtt{PWo}$ with an  associative species with restrictions structure.

\subsection{Species on packed words}

For packed words, we will mimic the framework produced above for permutations.
First, recall that a linear partial order $\leq$ on a set $I$ is an order inherited by a surjective map $I \to [m]$, called the rank of $\leq$, or $rk_{\leq}$.
We abuse notation and also call the integer $m$ the rank of the order $\leq$.

\begin{figure}[h]
	\centering
	\includegraphics[scale=1]{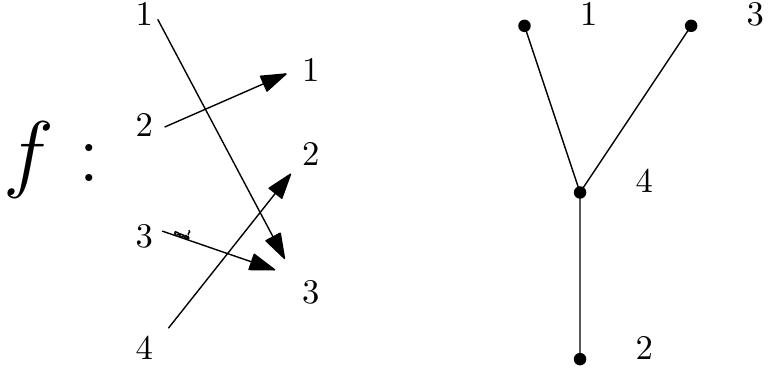}
	\caption{\textbf{Left:} The description of the function $f$. \textbf{Right:} The Hasse diagram of the linear order inherited by $f$. \label{fig:packedWordOrder}}
\end{figure}

For instance, if $f = \{ a\mapsto 3, b\mapsto 1, c\mapsto 3, d\mapsto 2\}$ is a surjective map $\{a, b, c, d\}\to[3]$, the inherited order is $\{b < d < \{a, c \}\}$ and has rank three.
Its Hasse diagram is presented in \cref{fig:packedWordOrder}.
In this way, a packed word $\omega$ on $I$ is a pair of orders $(\leq_P, \leq_V)$ where $\leq_P$ is a total order in $I$, and $\leq_V$ is a linear partial order on $I$.
In particular, note that any permutation on $I$ can be seen as a packed word on $I$, as any total order is a partial linear order.
This relates to the usual notion of packed words as a word in $[m]$ in the following way:
if we order the elements of $I = \{a_1 \leq_P \dots \leq_P a_k \}$, then the corresponding packed word is 
$$rk_{\leq_V}(a_1)rk_{\leq_V}(a_2) \dots rk_{\leq_V}(a_k) \, .$$

Conversely, for any packed word $\omega = p_1\dots p_k$, there are several pairs of orders $(\leq_P, \leq_V)$ that correspond to the word $\omega $, all of which are isomorphic.

\

Consider for instance the packed words $\omega_1 = 13123$ and $\omega_2 = 32413$.
These correspond to packed words on $\{a, b, c, d, e\}$, for instance $\omega_1$ corresponds to $(a > b > c > d > e, \{b, e\} > d > \{a, c\})$ and $\omega_2$ corresponds to $(a > b > c > d > e, c > \{a, e\} >  b > d)$.

If $I \hookrightarrow J$, by restricting the orders on $I$ to orders on $J$, we obtain a restriction to a packed word on $J$.
The resulting species with restrictions structure is denoted by $\mathtt{PWo}$.

\

It will be useful to represent packed words $\omega $ in $I$ as rectangle diagrams labeled by $I$.
This is done in the following way: let $1\leq m \leq |I|$ be the rank of $\omega$, we place the elements of $I$ in an $m \times |I|$ grid so that the elements are placed horizontally according to the $\leq_P$ order, and vertically according to the $\leq_V$ order.
For instance, the packed word $\omega_1 = 13123 = ( d > e > c > a > b, \{b, e\} > \{a\} > \{d, c\})$ in $\{a, b, c, d, e\}$ can be represented as 
\begin{equation}
\begin{array}{|c|c|c|c|c|}
	\hline   & e &   &   & b \\
    \hline   &   &   & a &   \\
    \hline d &   & c &   &   \\
    \hline 
\end{array}\, \, \, .
\end{equation}

In this way, there are $(n!) \times \sum_{m = 1}^n \sFunc(I, [m])$ elements in $\mathtt{PWo}[I]$, where $n = |I|$ and $\sFunc(A, B)$ counts the surjective functions from $A$ to $B$.
Because each packed word $\omega $ has an isomorphism class of size $n!$, there are $ \sum_{m = 1}^n \sFunc(I, [m])$ many packed words, which is expected.

\

\label{defin:pwo}
If we consider a packed word $\omega = (\leq_P, \leq_V) $ on a set $I$, we write $\mathbb{X}(\omega) = I$.
If $f:J \to I $ is an injective map, the preimage of each order $\leq_P, \leq_V$ is well defined.
Furthermore, the preimage of $\leq_P$ is a total order on $J$, whereas the preimage of $\leq_V$ is a linear order on $J$.
Thus, this defines the packed word $\mathtt{PWo}[f](\omega )$.
The $\ast $ operation on orders can be extended to a \textbf{diagonal sum} operation $\oplus $ on packed words.
This is usually referred to the shifted concatenation of packed words.
This endows $\mathtt{PWo}$ with an  associative species with restrictions structure.

\

\subsection{Relation between parking functions and labelled Dyck paths}

Let us first recall the definition of a parking function and of a Dyck path.

\begin{defin}[Parking function]
A parking function $\mathfrak{p} = a_1 \dots a_n$ is a sequence of integers in $[n]$ such that, after reordering $a^{(1)} \leq a^{(2)} \leq \dots \leq a^{(n)}$, we have $a^{(i)} \leq i$ for all $i$.
\end{defin}

Examples of parking functions are $12$, $131$ and $3114$.

\begin{defin}[Dyck path]
Given an $n\times n$ square grid, a \textbf{Dyck path} on this square is an edge path, from the lower left corner to the upper right corner, that is always above the main diagonal.
It is a classical result that the number of Dyck paths of size $n$ is enumerated by Catalan numbers.
\end{defin}

To describe species on parking functions we need to use the construction of parking functions as labelled Dyck paths (see for instance \cite{Loehr}).
This notion of species with restrictions will recover a notion of patterns in parking functions studied in \cite{adeniran2022pattern}.

\

Specifically, if $I$ is a set of size $n$, we label a Dyck path $\mathcal D$ on an $n\times n$ square grid by a function $f$ that assigns unique values on the set $I$ to each of the \textbf{up} segments of the Dyck path.
If we enrich $I$ with a total order $\leq$ in such a way that in each sequence of \textbf{up} segments, the labels arise in \textbf{increasing} order, Bergeron et.al. construct in \cite{BGLPV2021} a parking function $\mathfrak{p} = \mathfrak{p}(I, \mathcal D, f, \leq) $.
We recover here such construction, adapted to the language in this article, for convenience.

\

There are two steps in this construction.
The first is to realize the labelled Dyck path as a weak set composition $\opi$ of $I$ in $|I|$ parts. The second is to translate the weak set composition $\opi$ and the order $\leq $ in $I$ as a parking function.
We will use an example on $I = \{1, 2, 3, 4, 5\}$ to help us highlight the most important details of this construction, presented in \cref{fig:construction_parking}.
There, we have a Dyck path $\mathcal D$ and a corresponding assignment $f$ to each \textbf{up} segment.
We assume the usual integer order on $I$.

\

\textbf{In the first step} we group each of the labels that occur on the $i$-th vertical line in the $i$-th block of $\opi$.
This gives us a weak set composition of $I$ with exactly $|I|$ parts.
Notice that some parts may be empty sets, as it happens in the example given in \cref{fig:construction_parking}.

\textbf{In the second step} we read off the position of each of the labels of $I$, writing down in which part these occur. The resulting sequence is a parking function of size $|I|$ (see, for example, \cite{BGLPV2021}).
In the example in \cref{fig:construction_parking}, for instance, we notice that there are no entries in the second block, so $2$ does not occur in the corresponding parking function.
Because there are three elements in the first block, namely $245$, in the parking function $\mathfrak p $ the character $1$ arises three times, on the second, fourth and fifth position.

\begin{figure}
    \centering
    \includegraphics{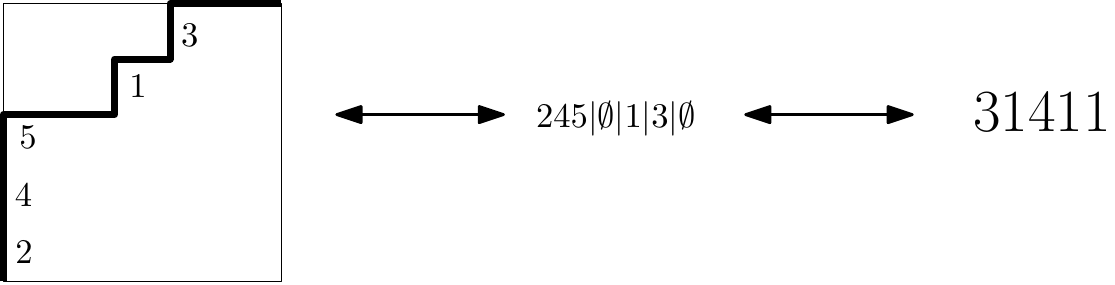}
    \caption{This is an example of a Dyck path labelled by the set $I =\{1, 2, 3, 4, 5\}$. We enrich this set with the usual order on the integers.}
    \label{fig:construction_parking}
\end{figure}

\subsection{Species on parking functions}

This forms the species of \textbf{parking functions}, $\mathtt{PF}$, by letting $\mathtt{PF}[I]$ be the collection of all parking functions $\mathfrak{p} = \mathfrak{p}(I, \mathcal D, f, \leq)$, that is Dyck paths $\mathcal D$ in an $|I|\times |I|$ grid, labelled by elements of $I$ along with an order $\leq$ of $I$.
We can give it a notion of species with restrictions: for each inclusion $\iota : I \hookrightarrow J $ and a labelled Dyck path $(J, \mathcal D, f, \leq)$ on $J$, the restriction $\mathtt{PF}[\iota](J, \mathcal D, f, \leq) = (I, \mathcal D|_I, f|_I, \leq|_I)$ has an intuitive meaning, except perhaps for
$\mathcal D|_I$, which we clarify in the following.
This is done via a notion of \textbf{tunnels}, introduced by Deutsch and Elizalde in \cite{elizalde2003simple}.
Specifically, the corresponding Dyck path is defined to be the restricted Dyck path by taking the \textbf{tunnels} labelled by elements in $I$, relabelling sequences of \textbf{up} segments if necessary to preserve the increasing property.
One can see in \cref{fig:restriction_parking} how this works on the example given above, as well as the corresponding parking function.

\begin{figure}[h]
\centering
    \subfloat[\centering The parking function 31411 and its restriction]{{\includegraphics[height=4.8cm]{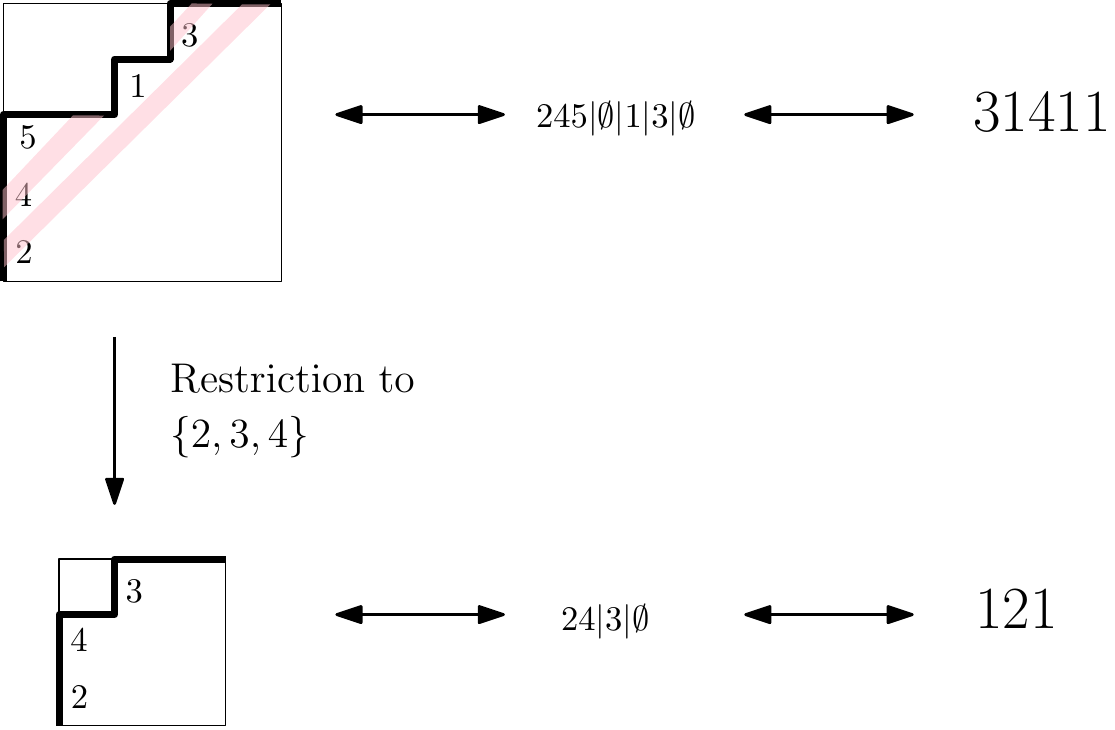} }}%
    \qquad
    \subfloat[\centering The parking function 41511 and its restriction]{{\includegraphics[height=4.8cm]{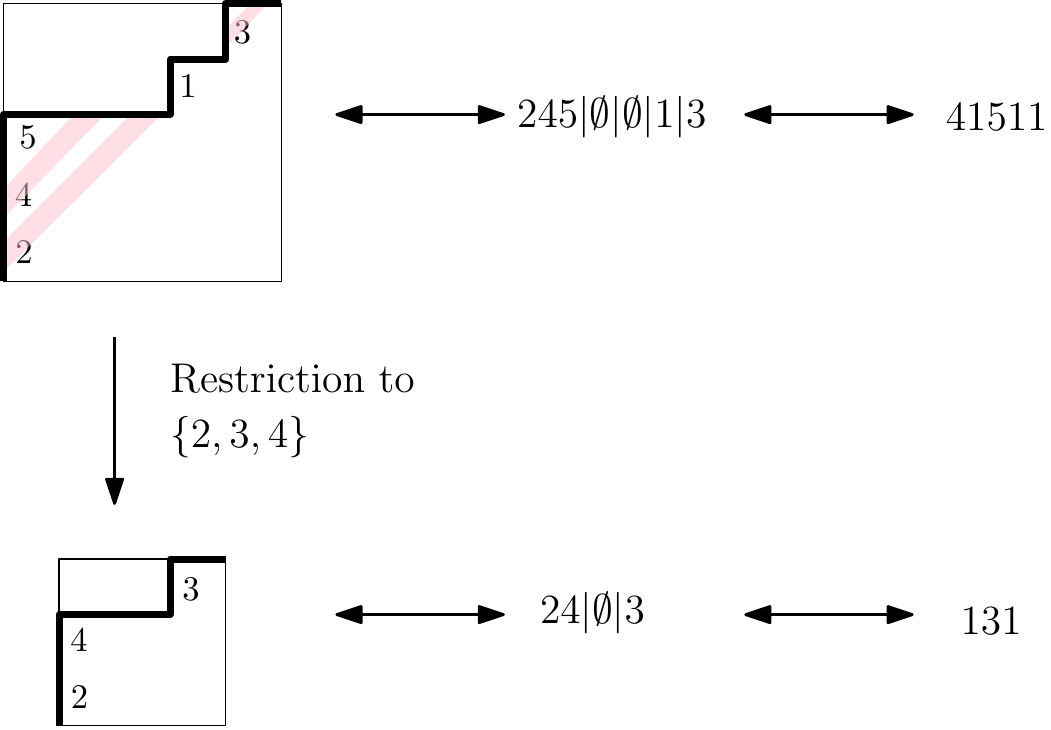}}}%
    \caption{\label{fig:restriction_parking}}%
\end{figure}

\

One defines a shifted shuffle $\oplus$ on parking functions, that results from a diagonal sum of its orders and Dyck paths.
The following claim is self evident and can be established by the same methods presented in \cite{Penaguiao2020}.

\begin{prop}[Species with restrictions on parking functions]
This forms a species with restrictions structure.
Furthermore, the shifted shuffle $\oplus $ endows $\mathtt{PF}$ with a monoid structure, and the resulting Hopf algebra $\mathcal A(\mathtt{PF})$ is free.
\end{prop}

For the last part, we observe that because $\mathtt{PF}$ is NCF (see below in \cref{defin:ncf}), we have from \cref{cor:freeNCF} that this algebra is free.

\begin{smpl}
The five smallest parking functions are $\emptyset, 1, 11, 12$ and $21$.
The sixteen parking functions of size three are displayed in \cref{fig:PF3}, along with its corresponding labelled Dyck paths.

For any parking function $\mathfrak p$ of size three, we have $\pat_{\emptyset}(\mathfrak p) = 1$ and $\pat_1(\mathfrak p) = 3$.
The values of $\pat_{11}, \pat_{12}$ and $\pat_{21}$ in pattern functions of size three are represented below in \cref{tab:PF3}.
There, one can also check the relation 
\[\pat_1^2 = \pat_1 + 2(\pat_{11} + \pat_{12} + \pat_{21}),\] 
predicted from \eqref{eq:prodrule}.
\end{smpl}
\begin{table}
\begin{tabular}{ c |  c c c c c c c c c c c c c c c c}
 0  & \small{111} & \small{112} & \small{121} & \small{211} & \small{113} & \small{131} & \small{311} & \small{122} & \small{212} & \small{221} & \small{123} & \small{132} & \small{213} & \small{312} & \small{231} & \small{321}\\ 
 11 & 3 & 2 & 2 & 2 & 1 & 1 & 1 & 1 & 1 & 1 & 0 & 0 & 0 & 0 & 0 & 0 \\  
 12 & 0 & 1 & 0 & 0 & 2 & 1 & 0 & 2 & 1 & 0 & 3 & 2 & 2 & 1 & 1 & 0 \\  
 21 & 0 & 0 & 1 & 1 & 0 & 1 & 2 & 0 & 1 & 2 & 0 & 1 & 1 & 2 & 2 & 3 \end{tabular}
\caption{\label{tab:PF3}Pattern functions evaluated at parking functions of length three.}
\end{table}

\begin{figure}
\centering
\includegraphics[scale=1]{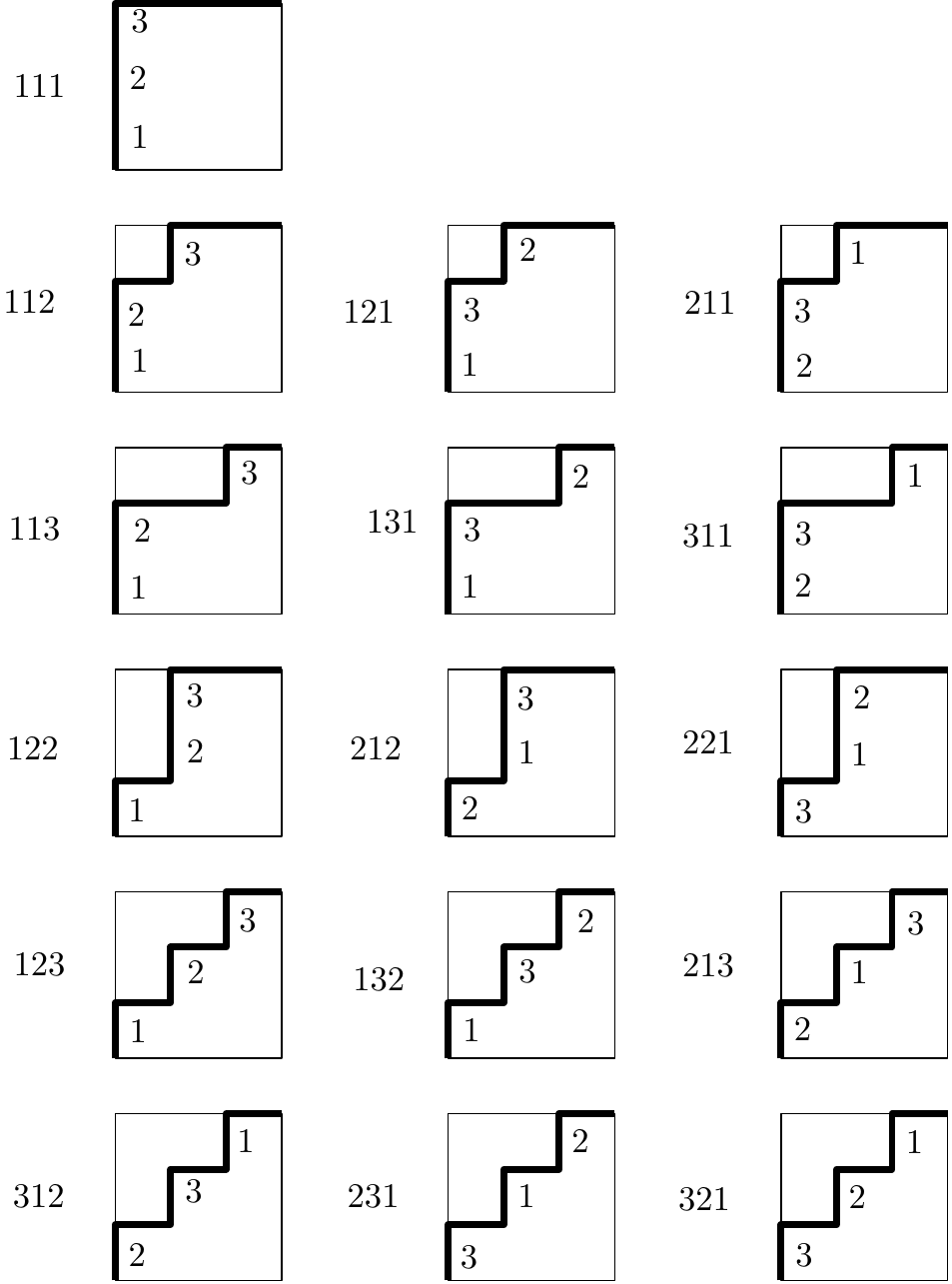}
\caption{\label{fig:PF3}}
\end{figure}

\

\section{The antipode formula for the pattern algebras \label{sec:formula_general}}
\

In this section we give a general formula for the antipode of a pattern algebra, whenever our connected species with restrictions is of the form $\mathtt{L} \times \mathtt{R}$. This antipode formula is a dual analogue of the antipode formula in \cite{BergeronBenedetti}, where a similar formula was obtained for linearized Hopf species.

\

The formula obtained is not cancellation free, but it servers as a starting platform to explore the cancellation free formulas for the cases presented above: permutations, packed words and parking functions. The requirement that our species with restrictions is divisible by $\mathtt{L}$ explains why no cancellation free formulas for other species with restrictions, for instance in marked permutations, introduced in \cite{Penaguiao2020}, were found.

\

We start by recalling Takeuchi's formula, from above in \cref{lm:takeuchi}.
If $H$ is a Hopf algebra such that $(  \iota  \circ\epsilon - \id_H)$ is $\star$-nilpotent, then 
$$S = \sum_{k\geq 0 }  ( \iota  \circ\epsilon- \id_H)^{\star k}\, . $$

Recall that for any pattern Hopf algebra $\mathcal A (\mathtt{h})$, $(\iota\circ \epsilon - \id_H)$ is $\star$-nilpotent. 

\

\subsection{$\mathtt{L}\times -$ species with restrictions}

In \cite[Corollary 3.4.]{Penaguiao2020} it was shown that in species with restrictions $\mathtt{h}$, any factorization into $\ast$-indecomposibles is unique up to order.
In some species with restrictions, we can also drop the ``up to order'' adjective, as there is exactly one factorization into $\ast$-indecomposibles full stop. We precise this in the \textit{non-commuting factorization} definition.
We will be using $\star$ for the operation on $\End (H)$ and we will be using $\ast $ for the associative structure on a species.

\begin{defin}[Non-Commuting Factorization on pattern Hopf algebras]\label{defin:ncf}
A monoidal species with restrictions is called a \textbf{non-commuting factorization} species, or simply an NCF species, if any element $x$ has a unique factorization into $\ast$-indecomposibles elements $x = x_1 \ast \dots \ast x_n$.
\end{defin}

\begin{lm}[Linear species with restrictions have NCF]
Let $\mathtt{R}$ be a connected species with restrictions.
Then $\mathtt{L} \times \mathtt{R}$ has NCF.
\end{lm}

In the following, we will refer to associative and connected species of the form $\mathtt{L} \times \mathtt{R}$ as \textbf{ordered species with restrictions}.

\begin{proof}
Let $x = x_1 \ast \dots \ast x_k $ and $ y =  y_1 \ast \dots \ast y_n$ such that $x \sim y$.
From \cite[Corollary 3.4.]{Penaguiao2020}, we know that $k = n$ and the multisets $\{x_i\}_{i=1}^k,  \{y_i\}_{i=1}^n$ are the same.
It remains to show that the factorizations order coincide and that we have $x_i \sim y_i$ for $i = 1, \dots , n$.
To that effect we act by induction, where $n = 1$ is trivial.

\

%For the induction step, let us consider the bijection $\sigma:[n] \to [n]$ such that $y_{\sigma(i)} \sim x_i$ given by \cite[Corollary 3.4.]{Penaguiao2020}.
We write $x_i = (l_i, p_i)\in (\mathtt{L}\times \mathtt{R})[I_i]$ and $y_i = (m_i, q_i)\in (\mathtt{L}\times \mathtt{R})[J_i]$.
Write $x = (l, p) \in (\mathtt{L}\times \mathtt{R})[I] $ and $y = (m, q) \in (\mathtt{L}\times \mathtt{R})[J]$.
Assume wlog that $|I_1 | \geq |J_1|$.

\

By hypothesis, we have that $x \sim y$, so there exists some bijection $\phi: J \to I$ such that $ (\mathtt{L}\times \mathtt{R})[\phi](x) = y$.
This bijection yields a correspondence between two linear orders $\mathtt{L} [\phi] (l_1 \ast \dots \ast l_n) = m_1\ast \dots \ast m_n $.
Observe that $I_1$ and $J_1$ are ideals in $l_1 \ast \dots \ast l_n$ and $m_1 \ast \dots \ast m_n$, respectively, as seen in \cref{prop:linorderideals}.
So either $J_1 = \phi (I_1)$ or $J_1 \subsetneq  \phi(I_1) $.
Assume for sake of contradiction that $J_1 \subsetneq  \phi(I_1)$, and consider the following factorization of $x_1 $:
\begin{equation}\label{eq:linearpats:non_trivial_fact}
x_1 = x|_{I_1} \sim y|_{\phi(I_1)} = y_1|_{\phi(I_1) \cap J_1} \ast (y_2 \ast \dots \ast y_n)|_{(J\setminus J_1) \cap \phi(I_1)}\, . 
\end{equation}

From $|I_1| > |J_1|$ and $\phi $ is a bijection, we have that $| (J\setminus J_1) | + | \phi(I_1) | = |J| - |J_1| + |I_1| > |J|$, so the intersection $(J\setminus J_1) \cap \phi(I_1)$ is non-empty.
On the other hand, $\phi(I_1) \cap J_1 = J_1$ is also non-empty.

\

Going back to \eqref{eq:linearpats:non_trivial_fact}, we get a non-trivial factorization of $x_1$, which contradicts the fact that we started with a factorization into indecomposables. Thus $|I_1 | = |J_1|$, which shows that $\phi(I_1) = J_1$.
Therefore,  $\phi(I\setminus I_1) = J \setminus  J_1$, $ \mathtt{R}[\phi] (p|_{I_1}) = q|_{J_1} $ and $ \mathtt{R}[\phi] (p|_{I \setminus I_1}) = q|_{J \setminus J_1} $.
We get that $x_1 \sim y_1$ and $x_2 \ast \cdots \ast x_n \sim y_2 \ast \cdots \ast y_n$, via $\phi$.
By induction hypothesis, this tells us that $x_2 \sim y_2, \cdots , x_n \sim y_n$, as desired.
\end{proof}

\begin{cor}\label{cor:freeNCF}
If $\mathtt{R}$ is an ordered species with restriction, then $\mathcal A(\mathtt{R})$ is free.
\end{cor}
This follows from \cite{Vargas}, after we recognize that $\mathtt{L} \times \mathtt{R}'$ has NCF from above.

\

The pattern Hopf algebra on permutations, on packed words and on parking functions all satisfy the NCF property.
In fact, permutations, packed words and parking functions are of the form $\mathtt{L} \times \mathtt{R}$.
This property allows for an easy manipulation of the coproduct, and results in a more tractable approach to the antipode formula.

\begin{defin}[Composition poset and cumulative sum]
Recall that we write $\mathcal C_n$ for the set of compositions of size $n$.
We will define $\mathbf{CS}$, a bijection between $\mathcal C_n $ and $2^{[n-1]}$, as follows.
If $\alpha =(\alpha_1, \dots, \alpha_{\ell} ) \in \mathcal C_n$, define $f_i = \sum_{j=1}^i \alpha_j$ and 
\begin{equation}
    \mathbf{CS}(\alpha) \coloneqq \{f_i\, | \, i = 1, \dots, \ell - 1\} \, .
\end{equation}
This bijection allows us to define an order $\leq $ in $\mathcal C_n$, via the pullback from the boolean poset in $2^{[n-1]}$.
This order can be defined as well as follows: we say that $\alpha \leq \beta$ if $\alpha$ arises from $\beta $ after merging and adding consecutive entries.
\end{defin}

\
Recall that a QSS $\III = (I_1, \dots , I_n)$ of $y\in \mathtt{R}[I]$ from $x_1, \dots, x_n$ satisfies $y|_{I_i} = x_i$ for all $i = 1, \dots , n$, and $I = \bigcup_i I_i$.

\begin{defin}[Compositions and QSS]
Consider again $\mathtt{R}$ an ordered species with restrictions.
Let $x \in \mathtt{R}[J], y \in  \mathtt{R}[I]$, and $x = x_1\ast \dots \ast x_n$ be the unique factorization of $x$ into indecomposables.
Further say that $y = (\leq_y, \iota)$, where $\leq_y$ is a linear order in $I$.
Let $\III = (I_1, \dots , I_n)$ be a QSS of $y$ from $x_1, \dots, x_n$ and consider a composition $\alpha \models n$.

\

Suppose that $ \mathbf{CS} (\alpha) = \{f_1, \dots, f_{\ell(\alpha) - 1} \}$ and use the convention that $f_0 = 0$ and $f_{\ell(\alpha)} = n$. 
Then we define for $i = 1, \dots, \ell(\alpha)$:
\[I^{\alpha}_i := I_{f_{i-1} + 1} \cup \dots \cup I_{f_i} \quad , \quad x^{\alpha}_i := x_{f_{i-1} + 1} \ast \dots \ast x_{f_i}.\]

\

For a partial order $\leq$ on a set $I$, and two sets $A, B \subseteq I$, we say that $A \lneq B$ if $A, B$ are disjoint and $a \leq b$ for any $a \in A$ and $b \in B$.

\

Two indices $i< j$ are said to be \textbf{merged} by $\alpha$ if there is some $k$ in $\{1, \dots, n\}$ such that $ f_{k-1} < i < j \leq f_k$.
We say that a QSS $\III$ of $y$ from $x_1, \dots , x_n$ is $\alpha $-stable if $(I^{\alpha}_i)_i \text{ is a QSS of $y$ from } (x^{\alpha}_i)_i$ and, whenever $x_i \sim x_j$ and $i < j$ are merged with $\alpha $, then $I_i \lneq_y I_j$.

\

Finally, we define 
\begin{equation}
   \mathcal I^{x, y}_{\III} \coloneqq \left\{ \alpha \models n \,\Big| \,\III \text{ is an $\alpha$-stable QSS of $y$ from } (x_i)_i \right\} \, . 
\end{equation}
\end{defin}

\

\begin{smpl}[$\alpha$-stable QSS on $\mathtt{PWo}$]
Consider the packed word $\rho = 21 \oplus 111 = 21333 = (1 < 2 < 3 < 4 < 5, 2 < 1 < \{3, 4, 5\})$ on the set $[5]$.
This packed word has three QSS from $21, 1, 11$, precisely $\III_1 =(12, 3, 45)$, $\III_2 =(12, 4, 35)$ and $\III_3 =(12, 5, 34)$.
All three are $(1, 1, 1)$-stable.

\

We now observe that $\III_1, \III_2$ and $\III_3$ are $(2, 1)$-stable, but neither $(1, 2)$-stable nor $(3)$-stable.
Indeed, $\III_1^{(1, 2)} = \III_2^{(1, 2)} = \III_3^{(1, 2)} = (12, 345)$, and $\rho|_{345} = 111$ is not $\rho|_{I_2} \oplus \rho|_{I_3}$ for any of the QSS.
On the other hand, $\III_1^{(2, 1)} = (123, 45)$, $\III_2^{(2, 1)} = (124, 35)$ and $\III_3^{(2, 1)} = (125, 34)$, in each case is easy to note that $\rho|_{I_1^{(2, 1)}} = 21\oplus 1$ and $\rho|_{I_2^{(2, 1)}} = 11$ for each of the QSS.

Therefore, in each case we get
$$\mathcal I_{\III}^{21\oplus 1\oplus 11, 21333} = \{(1, 1, 1), (2, 1)\}\, . $$
\end{smpl}

The following example portrays the significance of additional requirement that $I_i \lneq_y I_j$ whenever $x_i \sim x_j$ and $i, j$ are merged by $\alpha$, in the context of packed words.

\begin{smpl}[$\alpha$-stable on $\mathtt{PWo}$, with $x_i \sim x_j$]
Let us consider now the packed word $\rho = 2133 = (1 \leq_P 2 \leq_P 3 \leq_P 4, 2 \leq_V 1 \leq_V \{3, 4\} )$, we will be considering QSS of $\rho$ from $\omega_1 = 12, \omega_2 = 1, \omega_3 = 1$.
Namely, $\III_1 = (13, 2, 4)$ and $\III_2 = (13, 4, 2)$.
We have that $\rho|_{I_2 \cup I_3} = \rho|_{I_2} \oplus \rho|_{I_3} = 1 \oplus 1$ in either of the cases.
However, note that $\omega_2 \sim \omega_3$, and the composition $(1, 2)$ merges $2$ and $3$, so $(1, 2)$-stability requires $I_2 \lneq_P I_3$.
Indeed, because $2 \leq_P 4$, we have that $\III_1$ is $(1, 2)$-stable, whereas $\III_2$ is not $(1, 2)$-stable.
\end{smpl}

Define the composition $\mu_i \coloneqq (\underbrace{1, \dots , 1}_\text{$i-1$ times}, 2, 1, \dots, 1)$.
If $\III = (I_1, \dots, I_n)$  is a $\mu_i$-stable QSS, then $y|_{I_1 \cup I_{i+1}} = y|_{I_i} \ast y|_{I_{i+1}}$.
This motivates the following lemma:

\begin{lm}\label{obs:pwords-characterisation}
Let $\mathtt{R}$ be an ordered species with restriction.
Consider $x, y$ objects in $\mathtt{R}$, such that $y = (\leq_y, m)$ and $x = x_1\ast \dots \ast x_n$ a factorization into indecomposibles.
If $\III = (I_1, \dots, I_n)$ is a QSS of $y$ from $x_1, \dots x_n$ that is $\mu_i$-stable, then $I_i \lneq_y I_{i+1}$.
\end{lm}

We will observe latter that in the case of packed words, a stronger claim can be used to compute $\alpha$-stability.
See \cref{lm:QSSpackedWords}.

\begin{proof}
If $x_i \sim x_{i+1}$, the $\mu_i$ stability implies that $I_i \lneq_y I_{i+1}$, because $\mu_i$ merges $i$ and $i+1$.
We can now assume that $x_i \not\sim x_{i+1}$.
Let $X_j = \mathbb{X}(x_j)$ for $j = i, i+1$.
Note that because $y|_{I_j} \sim x_j$, we have that $|X_j| = |I_j|$.
The stability condition further gives us that $y|_{I_i \cup I_{i+1}} \sim x_i \ast x_{i+1}$.
Let $\phi : X_i \sqcup X_{i+1} \to  I_i \cup I_{i+1}$ be bijection such that $\mathtt{R}[\phi](y|_{ I_i \cup I_{i+1}}) = x_i \ast x_{i+1}$.

\

Because $\phi$ is a bijection, and $|X_j| = |I_j|$, this means that $I_i, I_{i+1}$ are disjoint.
We will consider the case that $|I_i| \geq |I_{i+1}|$ here.
The proof on the $|I_i| \leq |I_{i+1}|$ case can be done in a similar way.

\

Let $\inc $ be the injection $I_i \to I_i \cup I_{i+1}$.
Observe that
\begin{equation}\label{eq:lmonPWo}
    \begin{split}
        \mathtt{h}[\phi \circ \inc](y|_{I_i \cup I_{i+1}}) =& (x_i \ast x_{i+1})|_{\phi^{-1}(I_i)}\\
        \mathtt{h}[\phi](y|_{I_i}) =& (x_i)|_{X_i \cap \phi^{-1}(I_i)} \ast (x_{i+1})|_{X_{i+1} \cap \phi^{-1}(I_i)}\, .
    \end{split}
\end{equation}

However, $\mathtt{R}[\phi](y|_{I_i}) \sim y|_{I_i} \sim x_i$, which is $\ast$-indecomposible, we conclude that either $|X_i \cap \phi^{-1}(I_i)| = 0 $ or $|X_{i+1} \cap \phi^{-1}(I_i)| = 0$.
Assume the first for sake of contradiction, and because $\phi^{-1}(I_i) \subseteq X_i \cap X_{i+1}$, we have that $\phi^{-1}(I_i) \subseteq X_{i+1}$.
But $|\phi^{-1}(I_i)| = |I_i| \geq |I_{i+1}| = |X_{i+1}|$, therefore we have equality, that is $\phi^{-1}(I_i) = X_{i+1}$

\

This together with \eqref{eq:lmonPWo} and connectedness of $\mathtt{R}$ yields $\mathtt{h}[\phi](y|_{I_i}) = x_{i+1}$, which implies $x_i \sim x_{i+1}$, a contradiction.
We conclude that $|X_{i+1} \cap \phi^{-1}(I_i)| = 0$, and so $\phi^{-1}(I_i) = X_i$.
Because $I_i$ and $I_{i+1}$ are disjoint, it follows that $\phi^{-1}(I_{i+1}) = X_{i+1}$.

\

If we write $y = (\leq_y, m)$ and $x_i \ast x_{i+1} = (\leq_x, n)$, then $X_i \lneq_x X_{i+1}$ from \cref{prop:linorderideals}.
Because $\mathtt{R}[\phi^{-1}](x_i \ast x_{i+1}) = y|_{I_i \cup I_{i+1}}$, this gives us $\phi(X_i) \lneq_y \phi(X_{i+1})$, as desired.
\end{proof}

\

\begin{obs}
Let $\mathtt{R}$ be a species with restrictions, $y, x_1, \dots, x_n \in \mathcal G(\mathtt{R})$ and $\III$ a QSS of $y$ from $x_1\dots, x_n$.
Call $x \coloneqq x_1 \ast \dots \ast x_n$.
Then, $\mathcal I^{x, y}_{\III}$ has a unique maximal element, $\mathbb{1} = (1, \dots, 1)$.
\end{obs}

The following is the main theorem in this section

\begin{thm}\label{thm:general_antipode}
For a species with restrictions $\mathtt{R}$ that is a multiple of $\mathtt{L}$, and an element $x$, along with its factorization $x=x_1 \ast \dots \ast x_n$, we have the following antipode formula in $\mathcal A (\mathtt{R})$:

$$S(\pat_x) = \sum_y \pat_y \sum_{\substack{\III \text{ QSS of $y$}\\ \text{from }x_1, \dots , x_n}}  \sum_{\alpha \in \mathcal I^{x, y}_{\III}} (-1)^{\ell ( \alpha)} \, .$$
\end{thm}

\begin{proof}
\begin{align*}
\Delta^{\circ (k-1)} (\pat_x)
    =& \sum_{x = \chi_1 \ast \dots \ast \chi_k} \pat_{\chi_1}\otimes \dots \otimes \pat_{\chi_k}
\end{align*}
Because the species with restrictions $\mathtt{R}$ has NCF, the only ways to factorize $x$ into $k$ factors is to start from the original factorization $x = x_1 \ast \dots \ast x_n$ and bracket these factors into $k$ blocks, possibly empty, of $n$.
Therefore, we can enumerate these factorizations using weak compositions, as follows:

\begin{align*}
(\iota \circ \epsilon -  \id_{\mathcal A (\mathtt{R})})^{\star k} (\pat_x)
    =& (-1)^k  \mu^{\circ(k-1)}( \id_{\mathcal A (\mathtt{R})} - \iota \circ \epsilon)^{\otimes k} \Delta^{\circ(k-1)} (\pat_x)\\
    =& (-1)^k  \mu^{\circ(k-1)}( \id_{\mathcal A (\mathtt{R})} - \iota \circ \epsilon)^{\otimes k}\left(\sum_{\substack{\alpha\models^0 n \\ \ell(\alpha) = k}} \pat_{x^{\alpha}_1} \otimes \dots \otimes \pat_{x^{\alpha}_k} \right) \\
    =&  (-1)^k  \mu^{\circ(k-1)} \sum_{\substack{\alpha\models n \\ \ell(\alpha) = k}}  \pat_{x^{\alpha}_1} \otimes \dots \otimes \pat_{x^{\alpha}_k}\\
    =&  (-1)^k \sum_{\substack{\alpha\models n \\ \ell(\alpha) = k}}  \pat_{x^{\alpha}_1}  \cdots \pat_{x^{\alpha}_k}
\end{align*}

Note how we used that $(\id_{\mathcal A (\mathtt{R})} - \iota \circ \epsilon ) (\pat_x) = \mathbb{1}[x \neq 1]\pat_x$.
Takeuchi's formula gives:
%Thus, the theorem follows from Takeuchi's formula (\cref{lm:takeuchi}) after some clever rearrangements:

\begin{align*}
S(\pat_x)&= \sum_{k\geq 0 } (\iota \circ \epsilon -  \id_{\mathcal A (\mathtt{R})})^{\star k} (\pat_x) \\
         &= \sum_{\alpha \models n} (-1)^{\ell(\alpha)} \pat_{x^{\alpha}_1} \cdots \pat_{x^{\alpha}_{\ell(\alpha)}}\\
         &= \sum_{\alpha \models n } (-1)^{\ell(\alpha)} \sum_y \pat_y \binom{y}{x^{\alpha}_1,  \dots , x^{\alpha}_{\ell(\alpha)}}\\
         &= \sum_y \pat_y \sum_{\alpha \models n} \sum_{  \substack{\III \text{ QSS of $y$}\\ \text{from }x^{\alpha}_1,  \dots , x^{\alpha}_{\ell(\alpha)} }} (-1)^{\ell(\alpha)}\\
         &= \sum_y \pat_y \sum_{  \substack{\III \text{ QSS of $y$}\\ \text{from }x_1 , \dots ,  x_n }} \sum_{\alpha \in \mathcal I^{x, y}_{\III}} (-1)^{\ell(\alpha)}.
\end{align*}

All equalities but the last one are simple rearrangements.
To motivate the last equality, we will construct a series of bijections, for each composition $\alpha \models n$:
\begin{align*}
    \Phi_{\alpha}^{x, y} = \Phi_{\alpha} : \left\{ \JJJ \text{ QSS of $y$ from $x^{\alpha}_1, \dots, x^{\alpha}_{\ell(\alpha)}$}\right\} &\to \left\{ \III \text{ $\alpha$-stable QSS of $y$ from $x_1, \dots, x_n$}\right\}\, , \\
    (J_1, \dots, J_{\ell(\alpha)} ) &\mapsto (I_1, \dots , I_n).
\end{align*}

\

Recall that $y = (\leq_y, \iota)$ induces a linear order $m$ in $I$.
The sets $(I_i)_i$ are defined so that $J_s = I_{f_{s-1}+1} \uplus \dots \uplus I_{f_s}$, that $|I_i | = |x_i|$ and that $I_i \lneq_y I_j $ whenever $f_{s-1} < i < j \leq f_{s}$.
There is clearly a unique way to choose such $(I_i)_i$.
This is easily seen to be an $\alpha$-stable QSS of $y$ from $x_1, \dots, x_n$.

\

Inversely, we define:
\begin{align*}
    \Psi_{\alpha}^{x, y} = \Psi_{\alpha}  : \left\{ \III \text{ $\alpha$-stable QSS of $y$ from $x_1, \dots, x_n$}\right\} &\to \left\{ \JJJ \text{ QSS of $y$ from $x^{\alpha}_1, \dots, x^{\alpha}_{\ell(\alpha)}$}\right\} \, , \\
    (I_1, \dots , I_n) &\mapsto (I_1^{\alpha}, \dots, I^{\alpha}_{\ell(\alpha)} ),
\end{align*}
where we recall that $I^{\alpha}_i = I_{f_{s-1}+1} \cup \dots  \cup I_{f_s}$.

\

The maps $\Psi_{\alpha}$ and $\Phi_{\alpha}$ are easily seen to be inverses of each other, establishing the bijection and the desired result.
\end{proof}

\

%\begin{smpl}[Example in permutation patterns]
%Consider $\sigma = 312$ and $\pi = 123 = 1\oplus 1 \oplus 1$.
%We will compute the antipode of $\pat_{\sigma}$ using this formula, by computing $\mathcal I^{\sigma, \pi}_{\III}$ for each QSS $\III$.
%
%\end{smpl}

\begin{prop}[Filtered structure of $\mathcal I$]\label{prop:filter_structure_I}
For an ordered species with restrictions $\mathtt{R}$, and an elements $x\in  \mathtt{R}[J], y\in \mathtt{R}[I]$, along with a factorization into $\ast$-indecomposibles $x = x_1 \ast \dots \ast x_n$ and a QSS $\III$ of $y$ from $x_1, \dots, x_n$, we have that $\mathcal I^{ x, y}_{\III}$ is a filter.
That is, if $\alpha \in \mathcal I^{ x, y}_{\III}$ and $\beta \geq \alpha$ then $\beta \in \mathcal I^{ x, y}_{\III}$.

\
Furthermore, if $\mathcal I^{ x, y}_{\III}$ has a unique minimal element distinct from $\mathbb{1}$, then
$$\sum_{\alpha \in \mathcal I^{x, y}_{\III}} (-1)^{\ell(\alpha)} = 0 \, .$$
\end{prop}

\begin{proof}
Say that $y = (\leq_y, \iota)$.
To establish the first fact, assume that $\alpha \in \mathcal I^{x, y}_{\III}$ and let $\beta \geq \alpha $.
Because $\beta \geq \alpha$, if $\beta $ merges $i < j$ then $\alpha $ merges $i < j$, thus we have that $I_i \lneq_y I_j$.
It remains to show that $(I^{\beta}_i)_i$ is a QSS of $y$ from $(x^{\beta}_i )_i$, or equivalent that $y|_{I_i^{\beta}} \sim x_i^{\beta}$.

\

However, because $\beta \geq \alpha$, for each $i$ we have $I^{\beta}_i \subseteq I^{\alpha}_j$ so 
\[y|_{I^{\beta}_i} = {\Big(}y|_{I^{\alpha}_j}{\Big )}|_{I^{\beta}_i} =x^{\alpha}_j |_{I^{\beta}_i} = x^{\beta}_i \, .\]

\

For the concluding part, we simply observe that if $\mathcal I^{x, y}_{\III}$ has a unique minimal element $\mathbb{0}$, then it is the interval of a boolean poset $\mathcal I^{x, y}_{\III} = [\mathbb{0}, \mathbb{1}]$.
%Furthermore, we have that $\mathbb{0} \neq \mathbb{1}$.
Let $K $ be the collection of sets $A \subseteq [n-1]$ that contain $X \coloneqq \mathbf{CS}(\mathbb{0})$.

\

Consider $\zeta: K \to K$.
Choose $\chi \in [n-1] \setminus X$, if $X \neq [n-1]$, and consider the following symmetric difference:
\[\zeta (A) = A \,\Delta\, \{ \chi \}\]

If no such $\chi $ exists, define $\zeta(A) = A$.
This is an involution on $K$, % subsets of $[n-1]$ that contain $X$, which 
which corresponds to a sign-reversing involution in compositions in $[\mathbb{0} , \mathbb{1}]$:
%It can be seen that this is indeed sign-reversing, so we conclude that 
$$\sum_{\alpha \in \mathcal I^{x, y}_{\III}} (-1)^{\ell(\alpha)} =  \mathbb{1}[[n-1] \setminus X \text{ is empty } ] = \mathbb{1}[\mathbb{1} = \mathbb{0}] \, .$$
This concludes the proof.
\end{proof}

%We use the bijection $\mathbf{CS}$ between compositions and subsets of $\{1, \dots , n-1\}$ to compute the following sum.
%Let $X \coloneqq \mathbf{CS}(\mathbb{0} )$, then we can consider $\zeta : \mathcal K \to \mathcal K$

%\begin{align*}
%\sum_{\alpha \in \mathcal I^{y, x}_{\III}} (-1)^{\ell(\alpha)} &= \sum_{X \subseteq Y \subseteq [n-1]} (-1)^{|Y| +1} = \sum_{k = |X|}^{n-1} \sum_{\substack{X \subseteq Y \subseteq [n-1]\\ |Y| = k}} (-1)^{k +1} \\
%&= \sum_{k = |X|}^{n-1}\binom{n-|X| - 1}{k - |X|} (-1)^{k +1} = \sum_{k = 0}^{n- |X|-1}\binom{n-|X| - 1}{k} (-1)^{k + |X| +1}  \\
%&= (1 -1)^{n - |X| - 1 + |X| +1} = 0 \, .\\
%\end{align*}

%Notice that because $X \neq [n-1]$, we can use the Newton identity with exponent $n-|X| - 1$.

\

The consequence is, for pattern algebras, that computing the antipode corresponds to find which $\mathcal I_{\III}^{x, y}$ have a unique minimal element.
In this case, the antipode formula reduces to counting how many of these minimal elements are the composition $\mathbb{1}$.
These will contribute with $(-1)^n$ to the total coefficient of $\pat_y$.

\

\section{The antipode formula on special cases\label{sec:formula_pp}}
\subsection{The antipode formula for the pattern algebra on packed words\label{sec:formula_packed}}
%Let $\rho = (\leq_P, \leq_V)$ be a packed word.
For a partial order $\leq$ on a set $I$, and two sets $A, B \subseteq I$, recall that we say that $A \lneq B$ if $A\cap B = \emptyset$ and $a \leq b$ for any $a \in A$ and $b \in B$.
We recall the definition of non-interlacing QSS on packed words.

\begin{defin}[Interlacing QSS on packed words]
Let $\rho, \omega_1, \dots, \omega_n$ be packed words, where $\rho = (\leq_P, \leq_V)$ is a packed word on $I$.
Let $\III = (I_1, \dots, I_n)$ be a QSS of $\rho$ from $\omega_1, \dots, \omega_n$.
We say that this QSS is \textbf{non-interlacing} if there exists some $i = 1, \dots, n-1$ such that $I_i \lneq_P I_{i+1}$ and $I_i \lneq_V I_{i+1}$.
If no such $i$ exists, we say that the QSS is \textbf{interlacing}.

\

Additionally, let $ \bigl[\!\begin{smallmatrix} \rho  \\ \omega_1, \dots, \omega_n \end{smallmatrix}\!\bigr]$ be the number of interlacing QSS of $\rho$ from $\omega_1, \dots, \omega_n$.
\end{defin}

\

Our goal in this section is to analyse $\mathcal I^{\omega, \rho}_{\III}$, show that it always has a unique minimal element, and that it is precisely $\mathbb{1}$ whenever $\III$ is interlacing.
Recall that $\mu_i = (\underbrace{1, \dots , 1}_\text{$i-1$ times}, 2, 1, \dots, 1)$.

\

\begin{lm}\label{lm:QSSpackedWords}
Let $\omega, \rho = (\leq_P, \leq_V)$ be packed words, such that $\omega = \omega_1 \oplus \dots \oplus \omega_n$ is its factorization into $\oplus$-indecomposibles.
Let $\III$ be a QSS of $\rho$ from $\omega_1, \dots , \omega_n$.
Then $\III$ is $\mu_i$-stable if and only if $I_i \lneq_P I_{i+1}$ and $I_i \lneq_V I_{i+1}$.
\end{lm}

\begin{proof}
Let us take care of the \textbf{forward direction} first.
From \cref{obs:pwords-characterisation}, we have that $I_i \lneq_P I_{i+1}$, therefore $I_i$ and $I_{i+1}$ are disjoint and $I_i$ is the unique ideal of $(\leq_P)|_{I_i \cup I_{i+1}}$ of size $|I_i|$.
We have that $\rho|_{I_i \cup I_{i+1}} \sim \omega_i \oplus \omega_{i+1}$, so the unique ideal of $(\leq_V)|_{I_i \cup I_{i+1}}$ of size $|I_i|$ is also an ideal of $(\leq_P)|_{I_i \cup I_{i+1}}$, so it has to be $I_i$.
We conclude that $I_i$ is an ideal of $(\leq_V)|_{I_i \cup I_{i+1}}$, so $I_i \lneq I_{i+1}$, concluding this part of the proof.

\

For the \textbf{backwards direction}, if $I_i \lneq_P I_{i+1}$ then $I_i$ and $I_{i+1}$ are disjoint.
Therefore, by the definition of $\ast $ in partial orders, 
$$((\leq_P)|_{I_i \cup I_{i+1}}) = ((\leq_P)|_{I_i} \ast (\leq_P)|_{I_{i+1}})\, . $$
Similarly for $\leq_V$.
So we conclude that 
$$\rho|_{I_i \cup I_{i+1}} = ((\leq_P)|_{I_i} \ast (\leq_P)|_{I_{i+1}}, (\leq_V)|_{I_i} \ast (\leq_V)|_{I_{i+1}}) = \rho|_{I_i} \oplus \rho|_{I_{i+1}}\, .  $$
If $\omega_i \sim \omega_{i+1}$, we also have that $I_i \lneq_P I_{i+1}$, so $\III$ is $\mu_i$-stable.
\end{proof}

\

\begin{lm}\label{lm:minpacked}
Let $\omega$ be  packed word, along with $\omega = \omega_1 \oplus \dots \oplus \omega_n$, its unique factorization into $\oplus$-indecomposible packed words.
Let $\rho$ be another packed word, and $\III$ a QSS of $\rho$ from $\omega_1, \dots , \omega_n$. Then $\mathcal I^{\omega, \rho}_{\III}$ has a unique minimal element and, therefore, it is an interval in $\mathcal C_n$.
\end{lm}

\begin{proof}
%Recall that $\rho$ is a pair of orders in $I$.
%Notice that from the factorization $\omega = \omega_1 \oplus \dots \oplus \omega_n$
We give a concrete description of $\mathcal I^{\omega, \rho}_{\III}$.
Consider the following set:
$$J = \{i \in [n-1] | I_i \lneq_P I_{i+1} \text{ and } I_i \lneq_V I_{i+1} \} = \{i \in [n-1] | \, \III \text{ is $\mu_i$-stable } \} \, ,$$
where the second equality is due to \cref{lm:QSSpackedWords}.
Let $\beta = \mathbf{CS}^{-1}( J)$. 
We claim that $\beta \in \mathcal I^{\omega, \rho}_{\III}$ and that this is its smallest element.

\
\begin{itemize}
\item {\bf That $\beta$ is in $\mathcal I^{\omega, \rho}_{\III}$} we prove now.\\
Indeed, we just need to establish that for each $I^{\beta}_i$ we have that $\rho|_{I^{\beta}_i} = x^{\beta}_i$.
Because $I^{\beta}_i = I_{f_{i-1} + 1} \cup \dots \cup I_{f_i}$, and $I_{f_{i-1} + 1} \lneq_P \dots \lneq_P I_{f_i}$, $I_{f_{i-1} + 1} \lneq_V \dots \lneq_V I_{f_i}$, we get that 
$$\rho|_{I^{\beta}_i} = \rho|_{I_{f_{i-1} + 1}} \oplus \dots \oplus \rho|_{I_{f_i}} = x_{f_{i-1} + 1} \oplus \dots\oplus x_{f_i} = x^{\beta}_i\,  ,$$

\

    \item {\bf That $\beta$ is the smallest element in $\mathcal I^{\omega, \rho}_{\III}$} we prove now. \\
Let $\alpha $ be a composition of $n$ such that $\alpha \not \geq \beta$.
Assume that $\alpha \neq \beta$. 
Then $J^c \cap \mathbf{CS}(\alpha)\neq \emptyset $, so pick some $i \in  J^c \cap \mathbf{CS}(\alpha) $.
Because $i \not\in J$, is such that $ I_i \not\lneq_P I_{i+1} \text{ or } I_i \not\lneq_V I_{i+1} $.

\

If $\omega_i \not\sim \omega_{i+1}$, one can see that $\rho|_{I_i \cup I_{i+1}} \neq \omega_i \oplus \omega_{i+1}$.
If $\omega_i \sim \omega_{i+1}$, stability would require $I_i \lneq_P I_{i+1}$, so $I_i \not\lneq_V I_{i+1}$, and we conclude again that $\rho|_{I_i \cup I_{i+1}} \neq \omega_i \oplus \omega_{i+1}$.

\

So we can never have $\rho|_{I^{\alpha}_j} = \omega^{\alpha}_j$ for $j$ such that $\omega^{\alpha}_j$ includes both $\omega_i \oplus \omega_{i+1}$ in its factorization.
If $\omega_i \sim \omega_{i+1}$, stability would require $I_i \lneq_P I_{i+1}$, so $I_i \not\lneq_V I_{i+1}$, and we conclude again that $\rho|_{I_i \cup I_{i+1}} \neq \omega_i \oplus \omega_{i+1}$.
This is the contradiction that we are aiming for.
\end{itemize}
With the construction of the minimal element, we have that $\mathcal I^{\omega, \rho}_{\III}$.
\end{proof}

\

\begin{thm}\label{thm:antipode_packed}
Let $\omega $ be a packed word, and $\omega = \omega_1 \oplus \dots \oplus \omega_n$ be its decomposition into $\oplus$-indecomposible packed words.
Then, on the pattern Hopf algebra of packed words, we have the following cancellation free and grouping free formula:
$$S(\pat_{\omega}) = (-1)^n  \sum_{\rho} \bigl[\!\begin{smallmatrix} \rho  \\ \omega_1, \dots, \omega_n \end{smallmatrix}\!\bigr] \pat_{\rho}  \, .$$
\end{thm}

\begin{proof}
From \cref{thm:general_antipode}, we only need to establish that

\begin{equation}\label{eq:packed_alternating_sum}
\sum_{\alpha\in \mathcal I^{\rho, \omega}_{\III}} (-1)^{\ell (\alpha)} = (-1)^n \mathbb{1}[\III \text{ is interlacing QSS of $\rho$ from } \omega_1, \dots, \omega_n]\, .      
\end{equation}

\

Further, from \cref{prop:filter_structure_I} we know that the sum $\sum_{\alpha\in \mathcal I^{\rho, \omega}_{\III}}  (-1)^{\ell (\alpha)} $ vanishes whenever $\mathcal I^{\rho, \omega}_{\III}$ is an interval with more than one element.
From \cref{lm:minpacked}, we know that $\mathcal I^{\rho, \omega}_{\III}$ is indeed an interval.
The minimal interval is $\mathbb{1}$ if and only if $\III$ is an interlacing QSS from \cref{lm:minpacked}.
This concludes the proof.
\end{proof}

\

%From \cref{lm:minpacked}, we know that $\mathcal I^{\rho, \omega}_{\III}$ is indeed an interval, so the sum on the LHS of \eqref{eq:packed_alternating_sum} is zero except when $\mathbb{1} = \mathbf{CS}^{-1}([n-1]\setminus J)$, that is, when 
%$$\{i \in [n-1] | I_i <_P I_{i+1} \text{ and } I_i <_V I_{i+1} \} = \emptyset \, .$$

%This is precisely when $\III$ is an interlacing QSS.

\

Notice that this proof hides a sign-reversing involution in it.
Specifically, it was used in establishing \cref{prop:filter_structure_I}.

\

\subsection{The antipode formula for the pattern algebra on permutations\label{sec:formula_permutation}}
We start by recalling the definition of interlaced QSS on permutations.

\begin{defin}[Interlacing QSS on permutations]
Let $\sigma, \pi_1, \dots, \pi_n$ be permutations, where $\sigma = (\leq_P, \leq_V)$ is a permutation on $I$.
Let $\III = (I_1, \dots, I_n)$ be a QSS of $\sigma$ from $\pi_1, \dots, \pi_n$.
We say that $\III$ is \textbf{non-interlacing} if there exists $i = 1, \dots, n-1$ such that $I_i \lneq_P I_{i+1}$ and $I_i \lneq_V I_{i+1}$.
If no such $i$ exists, we say that the QSS is \textbf{interlacing}.

\

Additionally, let 
$ \bigl[\!\begin{smallmatrix} \sigma \\ \pi_1, \dots, \pi_n \end{smallmatrix}\!\bigr]$ be the number of interlacing QSS of $\sigma$ from $\pi_1, \dots, \pi_n$
\end{defin}

\

\begin{thm}\label{thm:antipode_perms}
Let $\pi $ be a permutation, and $\pi = \pi_1 \oplus \dots \oplus \pi_n$ be its decomposition into irreducible permutations.
Then, on the pattern Hopf algebra of permutations, we have the following cancellation free and grouping free formula:
$$S(\pat_{\pi}) = (-1)^n \sum_{\sigma} \bigl[\!\begin{smallmatrix} \sigma \\ \pi_1, \dots, \pi_n \end{smallmatrix}\!\bigr] \pat_{\sigma} \, .$$
\end{thm}

\

Although we can obtain the antipode formula by showing that a relevant poset of compositions is an interval (see \cref{lm:minpacked}), we present here a different proof.
Specifically, we will be leveraging a map $\mathcal A[\mathrm{inc}]: \mathcal A(\mathtt{PW}) \to \mathcal A(\mathtt{Per})$ and the previous result on packed words.

\

First, observe that any permutation is a packed word, because any pair of total orders $(\leq_P, \leq_V)$ is a packed word, that is, a pair of partial orders such that $\leq_P$ is a total order and $\leq_V$ is partial linear order.
This gives us an inclusion map $\mathrm{inc} : \mathtt{Per} \to \mathtt{PW}$ that preserves restrictions and the monoidal structure.
Therefore, this gives us a surjective Hopf algebra morphism $\mathcal A[\mathrm{inc}]: \mathcal A(\mathtt{PW}) \to \mathcal A(\mathtt{Per})$.

\begin{proof}
From \cref{thm:antipode_packed}, we know that for any permutation $\pi$ seen as a packed word we have that
$$S(\pat_{\pi}) = (-1)^n \sum_{\omega} \bigl[\!\begin{smallmatrix} \omega \\ \pi_1, \dots, \pi_n \end{smallmatrix}\!\bigr] \pat_{\omega}\, . $$

If $\omega$ is a packed word, we compute $\mathcal A[\,\mathrm{inc}] (\pat_{\omega}) = \pat_{\omega}\mathbb{1}[\omega \text{ is a permutation }]$.
Thus, applying $\mathcal A[\mathrm{inc}]$ to both sides of the equation above, we get the desired result.
Note that $\mathcal A[\mathrm{inc}]$ is a Hopf algebra morphism from \cref{thm:functoriality}, so it commutes with the antipode.
\end{proof}

\

\section*{Acknowledgments}

Both authors would like to thank Valentin F\'eray and Christophe Reutenauer, for helping this project to start.
The first author is also grateful to the SNF grant P2ZHP2 191301, while the second one is grateful to the Austrian Science Fund FWF, grant I
5788 PAGCAP.

\bibliographystyle{alpha}
\bibliography{Bibliography}

\newcommand{\etalchar}[1]{$^{#1}$}
\begin{thebibliography}{BDL{\etalchar{+}}23}

\bibitem[AA17]{aguiar2017hopf}
Marcelo Aguiar and Federico Ardila.
\newblock Hopf monoids and generalized permutahedra.
\newblock {\em arXiv preprint arXiv:1709.07504}, 2017.

\bibitem[ABF20]{albert2020two}
Michael Albert, Mathilde Bouvel, and Valentin F{\'e}ray.
\newblock Two first-order logics of permutations.
\newblock {\em Journal of Combinatorial Theory, Series A}, 171:105158, 2020.

\bibitem[AM10]{AM2010}
Marcelo Aguiar and Swapneel Mahajan.
\newblock {\em Monoidal functors, species and {H}opf algebras}, volume~29 of
  {\em CRM Monograph Series}.
\newblock American Mathematical Society, Providence, RI, 2010.
\newblock With forewords by Kenneth Brown and Stephen Chase and Andr\'{e}
  Joyal.

\bibitem[AM13]{AM2013}
Marcelo Aguiar and Swapneel Mahajan.
\newblock Hopf monoids in the category of species.
\newblock {\em Hopf algebras and tensor categories}, 585:17--124, 2013.

\bibitem[AP22]{adeniran2022pattern}
Ayomikun Adeniran and Lara Pudwell.
\newblock Pattern avoidance in parking functions.
\newblock {\em arXiv preprint arXiv:2209.04068}, 2022.

\bibitem[BB19]{BergeronBenedetti}
Carolina Benedetti and Nantel Bergeron.
\newblock The antipode of linearized hopf monoids.
\newblock {\em Algebraic Combinatorics}, 2(5):903--935, 2019.

\bibitem[BDL{\etalchar{+}}23]{BGLPV2021}
Nantel Bergeron, Rafael S~Gonz{\'a}lez D'Le{\'o}n, Shu~Xiao Li, CY~Amy Pang,
  and Yannic Vargas.
\newblock Hopf algebras of parking functions and decorated planar trees.
\newblock {\em Advances in Applied Mathematics}, 143:102436, 2023.

\bibitem[BS17]{BS2017}
Carolina Benedetti and Bruce~E. Sagan.
\newblock Antipodes and involutions.
\newblock {\em J. Combin. Theory Ser. A}, 148:275--315, 2017.

\bibitem[ED03]{elizalde2003simple}
Sergi Elizalde and Emeric Deutsch.
\newblock A simple and unusual bijection for dyck paths and its consequences.
\newblock {\em Annals of Combinatorics}, 7(3):281--297, 2003.

\bibitem[Foi22]{Foissy}
Lo{\"\i}c Foissy.
\newblock Bialgebras in cointeraction, the antipode and the eulerian
  idempotent.
\newblock {\em arXiv preprint arXiv:2201.11974}, 2022.

\bibitem[GR20]{GrinbergReiner}
Darij Grinberg and Victor Reiner.
\newblock {\em Hopf algebras in combinatorics}.
\newblock Mathematisches Forschungsinstitut Oberwolfach gGmbH, 2020.

\bibitem[HM12]{humpert2012incidence}
Brandon Humpert and Jeremy~L Martin.
\newblock The incidence hopf algebra of graphs.
\newblock {\em SIAM Journal on Discrete Mathematics}, 26(2):555--570, 2012.

\bibitem[Knu68]{Knuth}
Donald~E Knuth.
\newblock The art of computer programming, vol 1: Fundamental.
\newblock {\em Algorithms. Reading, MA: Addison-Wesley}, 1968.

\bibitem[KW66]{konheim1966occupancy}
Alan~G Konheim and Benjamin Weiss.
\newblock An occupancy discipline and applications.
\newblock {\em SIAM Journal on Applied Mathematics}, 14(6):1266--1274, 1966.

\bibitem[Loe11]{Loehr}
Nicholas Loehr.
\newblock {\em Bijective combinatorics}.
\newblock CRC Press, 2011.

\bibitem[LRV10]{linton2010permutation}
Steve Linton, Nik Ru{\v{s}}kuc, and Vincent Vatter.
\newblock {\em Permutation Patterns}, volume 376.
\newblock Cambridge University Press, 2010.

\bibitem[MR95]{MalvenutoReutenauer}
Claudia Malvenuto and Christophe Reutenauer.
\newblock Duality between quasi-symmetrical functions and the solomon descent
  algebra.
\newblock {\em Journal of Algebra}, 177(3):967--982, 1995.

\bibitem[Pen22]{Penaguiao2020}
Raul Penaguiao.
\newblock Pattern hopf algebras.
\newblock {\em Annals of Combinatorics}, pages 1--47, 2022.

\bibitem[PR04]{PR2004}
Fr{\'e}d{\'e}ric Patras and Christophe Reutenauer.
\newblock On descent algebras and twisted bialgebras.
\newblock {\em Moscow Mathematical Journal}, 4(1):199--216, 2004.

\bibitem[PS06]{PS2006}
Fr{\'e}d{\'e}ric Patras and Manfred Schocker.
\newblock Twisted descent algebras and the solomon--tits algebra.
\newblock {\em Advances in Mathematics}, 199(1):151--184, 2006.

\bibitem[PS08]{PS2008}
Fr{\'e}d{\'e}ric Patras and Manfred Schocker.
\newblock Trees, set compositions and the twisted descent algebra.
\newblock {\em Journal of Algebraic Combinatorics}, 28(1):3--23, 2008.

\bibitem[PV23]{penaguiao2023polynomial}
Raul Penaguiao and Yannic Vargas.
\newblock Polynomial invariants in permutation pattern algebras.
\newblock {\em to appear}, 2023.

\bibitem[QR18]{qiu2018patterns}
Dun Qiu and Jeffrey Remmel.
\newblock Patterns in words of ordered set partitions.
\newblock {\em arXiv preprint arXiv:1804.07087}, 2018.

\bibitem[Sch93]{Schmitt1993}
William~R. Schmitt.
\newblock Hopf algebras of combinatorial structures.
\newblock {\em Canad. J. Math.}, 45(2):412--428, 1993.

\bibitem[Sta75]{stanley1975combinatorial}
Richard~P Stanley.
\newblock Combinatorial reciprocity theorems.
\newblock In {\em Combinatorics}, pages 307--318. Springer, 1975.

\bibitem[Sto93]{Stover}
Christopher~R Stover.
\newblock The equivalence of certain categories of twisted lie and hopf
  algebras over a commutative ring.
\newblock {\em Journal of pure and applied Algebra}, 86(3):289--326, 1993.

\bibitem[Tak71]{Takeuchi1971}
Mitsuhiro Takeuchi.
\newblock Free {H}opf algebras generated by coalgebras.
\newblock {\em J. Math. Soc. Japan}, 23:561--582, 1971.

\bibitem[Var14]{Vargas}
Yannic Vargas.
\newblock Hopf algebra of permutation pattern functions.
\newblock {\em Discrete Mathematics \& Theoretical Computer Science}, 2014.

\bibitem[XY22]{xu2022cancellation}
Da~Xu and Houyi Yu.
\newblock On the cancellation-free antipode formula for the
  malvenuto-reutenauer hopf algebra.
\newblock {\em arXiv preprint arXiv:2208.06841}, 2022.

\bibitem[Zag09]{zagier2009one}
D~Zagier.
\newblock A one-sentence proof that every prime p= 1 (mod 4) is a sum of two
  squares.
\newblock {\em Biscuits of Number Theory}, 34:143, 2009.

\end{thebibliography}

\end{document}